\documentclass[a4paper, 12pt, reqno]{amsart}
\usepackage[left=1in, right=1in, top=1.2in, bottom=1.2in]{geometry}

\usepackage{times}
\usepackage{microtype}

\usepackage{commath, amssymb, amscd, mathrsfs, mathtools, dsfont}
\usepackage[all]{xy}
\usepackage{enumerate}

\usepackage[linktocpage]{hyperref}
\hypersetup{
    pdftitle={per-sub-var},            % title
    pdfauthor={Hu--Tan--Zhang},   % author
    colorlinks=true,        % false: boxed links; true: colored links
    linkcolor=red,          % color of internal links (change box color with linkbordercolor)
    citecolor=blue,         % color of links to bibliography
    filecolor=magenta,      % color of file links
    urlcolor=cyan           % color of external links
}

\newtheorem{theorem}{Theorem}[section]
\newtheorem{lemma}[theorem]{Lemma}
\newtheorem{prop}[theorem]{Proposition}

\theoremstyle{definition}

\newtheorem{example}[theorem]{Example}

\newtheorem{question}[theorem]{Question}

\newtheorem{remark}[theorem]{Remark}
\newtheorem*{notation}{Notation}

\newtheorem{hyp}{}

\newtheorem{hypp}{}

\theoremstyle{remark}

\newcommand{\cC}{\mathcal{C}}

\newcommand{\cK}{\mathcal{K}}
\newcommand{\cO}{\mathcal{O}}
\newcommand{\cP}{\mathcal{P}}

\newcommand{\bG}{\mathbf{G}}
\newcommand{\bC}{\mathbb{C}}
\newcommand{\bF}{\mathbb{F}}

\newcommand{\bZ}{\mathbb{Z}}
\newcommand{\bP}{\mathbb{P}}
\newcommand{\bQ}{\mathbb{Q}}
\newcommand{\bR}{\mathbb{R}}

\newcommand{\eps}{\epsilon}

\newcommand{\dom}{\operatorname{dom}}
\newcommand{\Pic}{\operatorname{Pic}}

\newcommand{\Diff}{\operatorname{Diff}}

\newcommand{\alb}{\operatorname{alb}}

\newcommand{\Aut}{\operatorname{Aut}}

\newcommand{\Exc}{\operatorname{Ex}}
\newcommand{\Per}{\operatorname{Per}}
\newcommand{\Bs}{\operatorname{Bs}}

\newcommand{\GL}{\operatorname{GL}}
\newcommand{\id}{\operatorname{id}}
\newcommand{\im}{\operatorname{Im}}
\newcommand{\Ker}{\operatorname{Ker}}

\newcommand{\Nef}{\operatorname{Nef}}

\newcommand{\NS}{\operatorname{NS}}
\newcommand{\Null}{\operatorname{Null}}

\newcommand{\rank}{\operatorname{rank}}
\newcommand{\Sing}{\operatorname{Sing}}
\newcommand{\SL}{\operatorname{SL}}

\newcommand{\Alb}{\operatorname{Alb}}

\newcommand{\isom}{\simeq}
\newcommand{\ratmap}{\dashrightarrow}

\newcommand{\num}{\equiv}

\DeclareFontFamily{U}{matha}{\hyphenchar\font45}
\DeclareFontShape{U}{matha}{m}{n}{
      <5> <6> <7> <8> <9> <10> gen * matha
      <10.95> matha10 <12> <14.4> <17.28> <20.74> <24.88> matha12
      }{}
\DeclareSymbolFont{matha}{U}{matha}{m}{n}
\DeclareFontSubstitution{U}{matha}{m}{n}
\DeclareMathSymbol{\abxcup}{\mathbin}{matha}{'131}

\begin{document}

\title[Periodic subvarieties of a projective variety]{Periodic subvarieties of a projective variety under the action of a maximal rank abelian group of positive entropy}

\author{Fei Hu}
\address{\textsc{Department of Mathematics} \endgraf \textsc{National University of Singapore, 10 Lower Kent Ridge Road, Singapore 119076}}
%\curraddr{\textsc{Department of Mathematics} \endgraf \textsc{University of British Columbia, 1984 Mathematics Road, Vancouver, BC V6T 1Z2, Canada}}
\email{\href{mailto:hf@u.nus.edu}{\tt hf@u.nus.edu}}

\author{Sheng-Li Tan}
\address{\textsc{Department of Mathematics} \endgraf \textsc{East China Normal University, 500 Dongchuan Road, Shanghai 200241, P.R. China}}
\email{\href{mailto:sltan@math.ecnu.edu.cn}{\tt sltan@math.ecnu.edu.cn}}

\author{De-Qi Zhang}
\address{\textsc{Department of Mathematics} \endgraf \textsc{National University of Singapore, 10 Lower Kent Ridge Road, Singapore 119076}}
\email{\href{mailto:matzdq@nus.edu.sg}{\tt matzdq@nus.edu.sg}}

\dedicatory{Dedicated to Professor Ngaiming Mok on the occasion of his 60th birthday}

\begin{abstract}
We determine positive-dimensional $G$-periodic proper subvarieties of an $n$-dimensional normal projective variety $X$ under the action of an abelian group $G$ of maximal rank $n-1$ and of positive entropy.
The motivation of the paper is to understand the obstruction for $X$ to be $G$-equivariant birational to the quotient variety of an abelian variety modulo the action of a finite group.
\end{abstract}



%\thanks{}

\maketitle

%\tableofcontents

%
%
%
%

\section{Introduction}\label{periodic-section-intro}

\noindent
We work over the field $\bC$ of complex numbers.
Let $X$ be a normal projective variety of dimension $n\ge 2$.
Denote by $\NS(X)\coloneqq\Pic(X)/\Pic^0(X)$ the {\it N\'eron--Severi group},
\index{N\'eron--Severi group}
i.e., the (finitely generated) abelian group of Cartier divisors modulo algebraic equivalence.
The rank of its torsion-free part is called the {\it Picard number} of $X$.
For a field $\bF = \bQ$, $\bR$ or $\bC$, ${\NS_\bF(X)}$ stands for ${\NS(X)} \otimes_{\bZ} \bF$.
The {\it first dynamical degree} of an automorphism $g\in\Aut(X)$ is defined as the {\it spectral radius} of its natural pullback action $g^*$ on $\NS_\bC(X)$: \index{dynamical degree!first dynamical degree}
$$d_1(g)\coloneqq\rho\big(g^*|_{\NS_\bC(X)}\big)\coloneqq\max\big\{|\lambda|:\lambda\textrm{ is an eigenvalue of } g^*|_{\NS_\bC(X)}\big\}.$$
Note that by Lemmas \ref{periodic-lemma-dyn-degree} and \ref{periodic-lemma-max-ev},
our definition of the first dynamical degree for possibly singular varieties coincides with
the usual one defined in Dinh--Sibony \cite[\S2.1]{DS04} for compact K\"ahler manifolds.
Also, by the fundamental work of Gromov \cite{Gromov03} and Yomdin \cite{Yomdin87},
the (topological) {\it entropy} of an automorphism $g\in\Aut(X)$ can be defined as
the logarithm of the spectral radius of the pullback action $g^*$ on the total cohomology ring $\oplus H^i(X,\bC)$.
Then by \cite[Corollary 2.2]{DS04}, an automorphism $g$ is of {\it positive entropy} (resp. {\it null entropy}), if and only if $d_1(g)>1$ (resp. $d_1(g)=1$).
\index{entropy!positive entropy} \index{entropy!null entropy}
See also \cite[Lemma 2.2]{Zhang09-JDG} and references therein.

For a subgroup $G\le\Aut(X)$, we define the {\it null-entropy subset} of $G$ as
$$N(G)\coloneqq\big\{g\in G : g \textrm{ is of null entropy, i.e., } d_1(g)=1\big\}.$$
We then call such $G\le\Aut(X)$ is of {\it positive entropy} (resp. {\it null entropy}),
if $N(G)=\{\id\}$ (resp. $N(G)=G$).
Assuming that $X$ is a compact K\"ahler manifold and $G$ is commutative, Dinh--Sibony \cite{DS04} showed that
$G$ contains a free abelian subgroup $G_1$ of positive entropy such that $\rank G_1 \le n-1$.
If $\rank G_1 = n-1$, then the null-entropy subset $N(G)$ is finite.

In general, we have a Tits alternative\footnote{Tits alternative is named after Jacques Tits, who first proved in \cite{Tits72} the deep and remarkable fact that
general linear groups satisfy this property.}
type result for any subgroup $G \le \Aut(X)$.
That is, either $G$ contains a subgroup isomorphic to the non-abelian free group $\bZ*\bZ$,
or $G$ is virtually solvable (i.e., a finite-index subgroup of $G$ is solvable).
\index{virtually!solvable}
In the latter case or when $G|_{\NS_\bC(X)}$ is virtually solvable,
there is a finite-index subgroup $G_1$ of $G$ such that $N(G_1)$ is a normal subgroup of $G_1$
and $G_1/N(G_1)$ is a free abelian group of rank $r\le n-1$.
We call this $r$ the {\it dynamical rank} of $G$ and denote it as $r = r(G)$,
which is independent of the choice of the finite-index subgroup $G_1$ of $G$.
See \cite{CWZ14, Dinh12, Zhang09-Invent} and references therein for details.

When the dynamical rank of $G$ is maximal (i.e., $r=n-1$), inspired by Dinh--Sibony \cite{DS04},
we expect in general that $N(G)$ is finite except the case when $X$ is an abelian variety.
This has been confirmed recently in \cite{DHZ15}.
Note that there indeed exist examples of abelian varieties and their quasi-\'etale quotients
admitting the action of commutative groups with maximal dynamical rank
(cf.~\cite[Example 4.5]{DS04}; see also our Example \ref{periodic-ex-rat-Q-ab}).
On the other hand, we are particularly interested in the geometry of those projective varieties
with the action of maximal rank abelian groups of positive entropy.
Along this direction, the third-named author obtained the following partial result already.

\begin{theorem}[{\cite{Zhang16-TAMS}}] \label{periodic-Zhang-main-thm}
Let $X$ be a normal projective variety of dimension $n\ge 3$ with at worst $\bQ$-factorial klt singularities,
and $G\le\Aut(X)$ such that the group $G^*\coloneqq G|_{\NS_\bC(X)}$ induced by the pullback action of $G$ on $\NS_\bC(X)$ is isomorphic to $\bZ^{\oplus n-1}$,
and every element of $G^*\setminus\{\id\}$ is of positive entropy.
Assume further that any one of the following conditions holds.
\begin{enumerate}[{\em (i)}]
  \item $X$ is not rationally connected\footnote{An algebraic variety $X$ is {\it rationally connected} (resp. {\it rationally chain connected}) in the sense of Campana and Koll\'ar--Miyaoka--Mori,
  if any two closed points on $X$ are contained in an irreducible rational curve (resp. a chain of rational curves).}.
  \item $X$ has no $G$-periodic\footnote{A Zariski closed subset $Z$ of $X$ is {\it $G$-periodic} if a finite-index subgroup of $G$ set-theoretically stabilizes $Z$.}
  proper subvariety of positive dimension.
\end{enumerate}
Then after replacing $G$ by a finite-index subgroup, $X$ is $G$-equivariant birational to a quasi-\'etale torus quotient.
\end{theorem}

By a {\it quasi-\'etale torus quotient}, \index{quasi-\'etale torus quotient}
we mean a quotient of an abelian variety $T$ by a finite group $F$,
which acts freely on $T$ outside a codimension-$2$ subset of $T$.
Note that such $T\to T/F$ is \'etale in codimension-$1$.
The purpose of this paper is to understand the obstruction for a normal projective variety $X$ with the action
of a maximal rank abelian group $G$ of positive entropy, to be $G$-equivariant birational to a quasi-\'etale torus quotient.
By virtue of \cite{Zhang16-TAMS} and \cite{DHZ15}, the remaining case we need to consider is the case when $X$ is rationally connected
or contains some non-trivial $G$-periodic proper subvariety of positive dimension.

It should be noted that there are rationally connected varieties which also have quasi-\'etale covers by abelian varieties (see Example \ref{periodic-ex-rat-Q-ab} for details).
The seemingly non-compatible rational connectivity and being quasi-\'etale torus quotient are allowed to co-exist, due to the existence of non-canonical klt singularities.
More precisely, a quasi-\'etale torus quotient which has at worst canonical singularities must be non-uniruled (even have vanishing Kodaira dimension by Koll\'ar--Larsen \cite[Theorem~10]{KL09}) and hence is not rationally connected.

Our main results are Theorems \ref{periodic-thm-ThmA'}, \ref{periodic-thm-ThmB'},
\ref{periodic-thm-ThmC} and Proposition \ref{periodic-prop-finite-periodic} below.

\begin{theorem} \label{periodic-thm-ThmA'}
Let $X$ be a normal projective variety of dimension $n\ge 2$, and $G\le\Aut(X)$ such that the following conditions are satisfied.
\begin{enumerate}[{\em (i)}]
  \item \label{periodic-thm-ThmA_i} $X$ has at worst $\bQ$-factorial klt singularities.
  \item \label{periodic-thm-ThmA_ii} $G|_{\NS_\bC(X)}$ is virtually solvable with maximal dynamical rank $r(G) = n-1$.
\end{enumerate}
Then after replacing $G$ by a finite-index subgroup, the following assertions hold.
\begin{enumerate}[{\em (1)}]
  \item \label{periodic-thm-ThmA_1} The union $\Per_+(X,G)$ of all positive-dimensional $G$-periodic proper subvarieties of $X$ is a Zariski closed proper subset of $X$.
  \item \label{periodic-thm-ThmA_2} Let $\Per_+(X,G) = Z_1 \cup Z_2 \cup \cdots \cup Z_m$ be the irreducible decomposition.
  Then either $Z_k$ is uniruled\footnote{A variety $V$ of dimension $d$ is {\it uniruled},
  if there exists a dominant rational map $\bP^1\times W \ratmap V$ for some variety $W$ of dimension $d-1$.
  Note that being uniruled is a birational property.}, or a finite-index subgroup of $G$ fixes $Z_k$ pointwise.
  \item \label{periodic-thm-ThmA_3} The Picard number $\rho(X) \ge n$.
  If $\rho(X) = n \ge 3$, then $X$ is $G$-equivariant birational to a quasi-\'etale torus quotient.\footnote{We remark that if the Picard number $\rho(X) > n^2$, then $X$ is not equal to a quasi-\'etale torus quotient.
  Indeed, $X$ is then not dominated by any abelian variety $T$ via a generically finite surjective morphism.
  This is because the Picard number $\rho(T) \le (\dim T)^2 = n^2$.}
  \item \label{periodic-thm-ThmA_4} Either $X$ is an abelian variety and hence has no positive-dimensional $G$-periodic proper subvariety, or $X$ has at most $\rho(X)-n$ distinct $G$-periodic prime divisors.
\end{enumerate}
\end{theorem}

The assertion (\ref{periodic-thm-ThmA_1}) of Theorem \ref{periodic-thm-ThmA'} follows from \cite[Proposition 3.11]{Zhang16-TAMS} or Proposition \ref{periodic-prop-Zhang-klt},
with the help of \cite[Theorem 4.1]{DHZ15} or Proposition~\ref{periodic-prop-finite-null} to deal with the solvable group case.
We include them here for the convenience of the reader.
Note that the condition (\ref{periodic-thm-ThmA_i}) of Theorem~\ref{periodic-thm-ThmA'}
(or Question \ref{periodic-qestion-QnA}) is not restrictive,
since we can always take a $G$-equivariant resolution due to Hironaka \cite{Hironaka77} and even assume that $X$ is smooth.
(See also Koll\'ar's book \cite[3.4.1, Proposition 3.9.1 and Theorem 3.36]{Kollar07} for a modern description.)
Meanwhile, the condition (\ref{periodic-thm-ThmA_ii}) is birational in nature (see Proposition \ref{periodic-prop-finite-null} and \cite[Lemma 3.1]{Zhang16-TAMS}).

If we assume further that $X$ contains a $G$-periodic non-uniruled prime divisor $D$,
we obtain a more clear geometric characterization of the pair $(X,D)$ by theorem below.
The main ingredient is to run a $G$-equivariant Minimal Model Program ($G$-MMP for short) developed in \cite{Zhang16-TAMS}.

\begin{theorem} \label{periodic-thm-ThmB'}
Let $X$ be a normal projective variety of dimension $n\ge 2$, and $G\le\Aut(X)$ such that the following conditions are satisfied.
\begin{enumerate}[{\em (i)}]
  \item \label{periodic-thm-ThmB_i} $G|_{\NS_\bC(X)}$ is virtually solvable with maximal dynamical rank $r(G) = n-1$.
  \item \label{periodic-thm-ThmB_ii} $X$ contains a $G$-periodic non-uniruled prime divisor $D$.
\end{enumerate}
Then after replacing $G$ by a finite-index subgroup, the following assertions hold.
\begin{enumerate}[{\em (1)}]
  \item \label{periodic-thm-ThmB_1} $X$ is rationally connected.
  \item \label{periodic-thm-ThmB_2} Let $Z_1 \cup Z_2 \cup \cdots \cup Z_m$ be the irreducible decomposition of the union of all positive-dimensional $G$-periodic proper subvarieties of $X$, with $Z_1 = D$.
  Then for $k\ge 2$, $Z_k$ is uniruled.
  In particular, every $G$-periodic prime divisor, other than $D$, is uniruled.
  \item \label{periodic-thm-ThmB_3} A finite-index subgroup of $G$ fixes $D$ pointwise.
\end{enumerate}
Furthermore, there exists a surjective in codimension-$1$ $G$-equivariant birational map $X \ratmap Y$ with $D_Y$ the push-forward of $D$, such that we have:
\begin{enumerate}[{\em (1)}] \setcounter{enumi}{3}
  \item \label{periodic-thm-ThmB_4} Every positive-dimensional $G$-periodic proper subvariety of $Y$ is contained in $D_Y$.
  In particular, the positive-dimensional part of $\Sing Y$ is contained in $D_Y$.
  \item \label{periodic-thm-ThmB_5} $K_Y + D_Y \sim_\bQ 0$ ($\bQ$-linear equivalence); both $K_Y$ and $D_Y$ are $\bQ$-Cartier; the pair $(Y, D_Y)$ and hence $Y$ both have at worst canonical singularities.
  \item \label{periodic-thm-ThmB_6} $D_Y$ has at worst canonical singularities and $K_{D_Y} \sim_\bQ 0$.
  \item \label{periodic-thm-ThmB_7} $-mD_Y|_{D_Y}$ is an ample Cartier divisor on $D_Y$ for some integer $m>0$.
\end{enumerate}
\end{theorem}

\begin{remark}
In dimension $2$, Theorem \ref{periodic-thm-ThmB'} means that if $X$ is a normal projective surface with an automorphism $g$ of positive entropy and $D$ is an irrational $g$-periodic curve, then $X$ is a rational surface, $D$ is an elliptic curve pointwise fixed by a power of $g$, and all other $g$-periodic curves are rational. See Lemma \ref{periodic-lemma-surface} and Remark \ref{periodic-remark-surface} for an elementary treatment. Also, there indeed exists an explicit example satisfying the conditions (\ref{periodic-thm-ThmB_i}) and (\ref{periodic-thm-ThmB_ii}) in Theorem \ref{periodic-thm-ThmB'}. See \cite[Theorem~2 or Example~3.3]{Diller11}. Indeed, in that example, $X$ is a smooth rational surface and $D$ is a smooth elliptic curve.
\end{remark}

\begin{example} \label{periodic-ex-rat-Q-ab}
Here we give examples of rationally connected varieties
which are quasi-\'etale torus quotients at the same time.

Let $E = E_{\zeta_m}\coloneqq \bC/(\bZ+\bZ\, \zeta_m)$ be the elliptic curve with period $\zeta_m \coloneqq \exp \frac{2 \pi i}{m}$
for some $m\in \{2, 3, 4, 6\}$ and $A = E^n \coloneqq E \times \cdots \times E$.
Let $\mu_m\coloneqq \langle \zeta_m \rangle$, the group of $m$-th roots of unity,
act on $E$ by multiplication and act on $A$ diagonally.
Then the quotient variety $X \coloneqq A/\mu_m$ has at worst $\bQ$-factorial klt singularities by \cite[Proposition 5.20]{KM98}.
Moreover, for any $2 \le n < m$, $X$ is a rationally connected variety, which is also a quasi-\'etale torus quotient.

Indeed, for any $\id \ne g\in \mu_m$, the Zariski closed set $E^g$ of $g$-fixed points in $E$ is a finite subset of $E$, so is $A^g$.
Thus the ramification locus of $A\to X$, as the union of all $g$-fixed points for all $g\ne \id$, is a finite set.
It follows that $X$ is a quasi-\'etale torus quotient when $n\ge 2$.

On the other hand, under the condition $n < m$, the {\it age} of the automorphism $[\zeta_m]$ at a fixed point $\mathrm{o} \in A^{[\zeta_m]}$
is $\frac{n}{m} < 1$ (see \cite[\S 2]{Reid02} for the definition of age).
Then the Reid--Tai criterion implies that $X$ has non-canonical singularities.
Take a resolution $X'$ of $X$.
Note that $K_X$ is $\bQ$-linearly equivalent to zero and $X$ has non-canonical klt singularities.
Thus $K_{X'}$ is not pseudo-effective.
Hence by \cite[Corollary 0.3]{BDPP13} $X'$ is uniruled, so is $X$.
Also, the natural $\SL_n(\bZ)$-action on $A$ descends to $X$.
As in \cite[Example 1.4]{Dinh12}, $\SL_n(\bZ)$ admits a free abelian subgroup isomorphic to $\bZ^{\oplus n-1}$
whose every non-trivial element $g$ has spectral radius $> 1$.
Thus the natural action of $g$ on $A$ is of positive entropy.
In other words, the dynamical rank of $\bZ^{\oplus n-1}|_A$ is maximal,
so is $\bZ^{\oplus n-1}|_X$ by Lemma \ref{periodic-lemma-dyn-degree}.
We then consider the so-called special maximal rationally connected (MRC) fibration
$X\ratmap Y$ of $X$ in the sense of Nakayama \cite[Theorems 4.18 and 4.19]{Nakayama10},
where the general fibres are rationally connected and
the $\bZ^{\oplus n-1}$-action on $X$ descends to a biregular action of $\bZ^{\oplus n-1}$ on $Y$.
The maximality of the dynamical rank implies that this special MRC fibration is trivial
(cf.~\cite[Lemma 2.10]{Zhang09-Invent}).
Thus $Y$ is a point and hence $X$ is rationally connected.
See also Koll\'ar--Larsen \cite[Corollary 25]{KL09} for another proof of the rational connectedness of $X$.

For instance, \cite[Example 4.2]{Zhang91} gives an explicit calculation for the case $(m, n)=(3, 2)$.
\end{example}

From Theorems \ref{periodic-thm-ThmB'} and \ref{periodic-Zhang-main-thm},
we see that the varieties containing $G$-periodic non-uniruled prime divisors
provide potential examples which are not $G$-equivariant birational to quasi-\'etale torus quotients in our setting.
Moreover, a positive answer to the question below roughly means that when $r(G) = n-1$ is maximal,
$X$ is $G$-equivariant birational to a quasi-\'etale torus quotient if and only if $X$ has no non-uniruled $G$-periodic prime divisor.

\begin{question} \label{periodic-qestion-QnA}
Let $X$ be a normal projecitve variety of dimension $n\ge 3$, and $G\le\Aut(X)$ such that the following conditions are satisfied.
\begin{enumerate}[(i)]
  \item $X$ has at worst $\bQ$-factorial klt singularities.
  \item $G|_{\NS_\bC(X)}$ is virtually solvable with maximal dynamical rank $r(G) = n-1$.
\end{enumerate}
Is it true that the following assertions hold?
\begin{enumerate}[(1)]
  \item Suppose that $X$ does not have any $G$-periodic non-uniruled prime divisor.
  Then $X$ is $G$-equivariant birational to a quasi-\'etale torus quotient.
  \item Suppose that $X$ has a $G$-periodic non-uniruled prime divisor.
  Then $X$ is not $G$-equivariant birational to a quasi-\'etale torus quotient.
\end{enumerate}
\end{question}

The theorem below gives an affirmative answer to Question \hyperref[periodic-qestion-QnA]{\ref{periodic-qestion-QnA} (2)},
see also Proposition \ref{periodic-prop-PropA}.
The implications (2) $\Longrightarrow$ (1) and (3) $\Longrightarrow$ (1) below are proved in \cite[Theorem 2.4]{Zhang16-TAMS}.
We include them here for the convenience of the reader.

\begin{theorem} \label{periodic-thm-ThmC}
Let $X$ be a normal projective variety of dimension $n\ge 3$,
and $G\le\Aut(X)$ such that $G|_{\NS_\bC(X)}$ is virtually solvable with maximal dynamical rank $r(G) = n-1$.
Consider the following conditions:
\begin{enumerate}[{\em (1)}]
  \item After replacing $G$ by a finite-index subgroup, $X$ is $G$-equivariant birational to a quasi-\'etale torus quotient $X'$.
  \item After replacing $G$ by a finite-index subgroup, $X$ is $G$-equivariant birational to a projecitve variety $X'$ with only klt singularities,
  such that $X'$ has no positive-dimensional $G$-periodic proper subvariety.
  \item After replacing $G$ by a finite-index subgroup, $X$ is $G$-equivariant birational to a projecitve variety $X'$ with a $G$-periodic divisor $D'$,
  such that $(X', D')$ is $\bQ$-factorial klt and $K_{X'}+D'$ is pseudo-effective.
  \item Every connected component of the union of all positive-dimensional $G$-periodic proper subvarieties of $X$ is rationally chain connected.
\end{enumerate}
Then the conditions {\em (1)}, {\em (2)} and {\em (3)} are equivalent, and imply the condition {\em (4)}.
\end{theorem}

The following proposition generalizes a well-known result on surface --
there are only finitely many $g$-periodic curves if $g$ is an automorphism of positive entropy on a projective surface.
We prove a result of this type up to dimension $3$ in the present paper.
Naturally, we would like to know whether it is still true in higher dimensions.

\begin{prop} \label{periodic-prop-finite-periodic}
Let $X$ be a normal projecitve variety of dimension $n=2$ or $3$, and $G\le \Aut(X)$ such that the following conditions are satisfied.
\begin{enumerate}[{\em (i)}]
  \item $X$ has at worst $\bQ$-factorial klt singularities.
  \item $G = \langle g_1, \dots, g_{n-1} \rangle \isom \bZ^{\oplus n-1}$ is of positive entropy.
\end{enumerate}
Then for any non-trivial $g\in G$, the following assertions hold.
\begin{enumerate}[{\em (1)}]
  \item \label{periodic-prop-finite-periodic_1} If $X$ is an abelian variety, then there is no $g$-periodic prime divisor.
  \item \label{periodic-prop-finite-periodic_2} If $X$ is not an abelian variety, then there are at most $\rho(X) - n$ distinct $g$-periodic prime divisors.
\end{enumerate}
\end{prop}

\section{Preliminary results}\label{periodic-section-pre}

\begin{notation}
We refer to Koll\'ar--Mori \cite{KM98} for the standard definitions, notation, and terminologies in birational geometry.
For instance, see \cite[Definitions 2.34 and 2.37]{KM98} for the definitions of {\it canonical} singularity,
Kawamata log terminal singularity ({\it klt}), divisorial log terminal singularity ({\it dlt}), and log canonical singularity ({\it lc}).
\index{singularity!canonical singularity} \index{singularity!klt singularity}
\index{singularity!dlt singularity} \index{singularity!lc singularity}

Let $X$ be a normal projective variety. \index{$\bQ$-factorial}
$X$ is called {\it $\bQ$-factorial}, if every integral Weil divisor $M$ on $X$ is $\bQ$-Cartier, i.e., $sM$ is a Cartier divisor for some integer $s \ge 1$.

Let $M$ be an $\bR$-Cartier divisor (an $\bR$-linear combination of integral Cartier divisors) on $X$.
We call $M$ is {\it nef}, \index{divisor!nef $\bR$-Cartier divisor}
if the intersection $M\cdot C \ge 0$ for every irreducible curve $C$ on $X$.
Denote by $\Nef(X)$ the closed cone of all nef $\bR$-Cartier divisors on $X$. \index{cone!nef cone}
We call $M$ is {\it pseudo-effective}, if it is contained in the closure of the cone of all effective $\bR$-divisors on $X$.
\index{divisor!pseudo-effective divisor}

For a birational map $f \colon X \ratmap Y$, denote its domain by $\dom f$.
Then for an irreducible subvariety $B$ of $X$ such that $f$ is defined at the generic point of $B$,
define the {\it birational transform} $f(B) \subset Y$ as the Zariski-closure of $f(B \cap \dom f)$ in $Y$.
\index{birational transform}
Then the {\it push-forward} $f_*B$ of $B$ under the birational map $f$ is defined (linearly) as follows:
\begin{equation*}
f_*B \coloneqq \left\{
\begin{array}{cl}
f(B), & \text{if $\dim f(B) = \dim B$}; \\
0, & \text{otherwise}.
\end{array}
\right.
\end{equation*}
In particular, if $f$ is isomorphic in codimension-$1$ and $D$ is a prime divisor, then $f_*D = f(D)$.

For an automorphism $g$ of $X$, we use $g|_X$ to emphasize that $g$ acts on $X$.
For a $g$-invariant subspace $V$ of some cohomology space $H^*(X, \bC)$,
we use $g^*|_V$ to denote the natural pullback action $g^*$ on $V$.
The {\it spectral radius} $\rho\big(g^*|_V\big)$ \index{spectral radius}
is the maximal absolute value of all eigenvalues of $g^*|_V$ as a linear transformation on $V$.
\end{notation}

The result below shows that our notion of the first dynamical degree of an automorphism as in the introduction is equivalent to the same one on its equivariant resolution,
and hence equivalent to the usual definition in Dinh--Sibony \cite[\S2.1]{DS04} by Lemma \ref{periodic-lemma-max-ev}.

\begin{lemma} \label{periodic-lemma-dyn-degree}
Let $X$ and $Y$ be two normal projective varieties of dimension $n \ge 2$,
and $f \colon X \to Y$ a $g$-equivariant generically finite surjective morphism.
Then we have $d_1(g|_X) = d_1(g|_Y)$.
In particular, $g|_X$ is of positive entropy (resp. null entropy) if and only if so is $g|_Y$.
\end{lemma}

\begin{proof}
The proof of \cite[Lemma 2.6]{Zhang09-JDG} also applies to our situation.
Let $W \to X\to Y$ be a $g$-equivariant resolution due to Hironaka \cite{Hironaka77}.
By using the Lefschetz hyperplane theorem (on $W$), we reduce to the surface case.
Then both $d_1(g|_X)$ and $d_1(g|_Y)$ are equal to $d_1(g|_W)$.
\end{proof}

Recall that for a compact K\"ahler manifold $X$,
the {\it first dynamical degree} $d_1(g)$ of a surjective endomorphism $g$ of $X$ is defined as
the spectral radius of the pullback action $g^*$ on $H^{1,1}(X,\bR)$ (cf.~\cite[\S A.2]{NZ09}).
The following lemma asserts that for smooth projective varieties these two definitions of $d_1$
(another one given in the introduction) for endomorphisms or automorphisms coincide.

\begin{lemma} \label{periodic-lemma-max-ev} %~
\begin{enumerate}[{\em (1)}]
  \item Let $(X, \omega)$ be a compact K\"ahler manifold of dimension $n$, and $g$ a surjective endomorphism of $X$.
  Let $V$ be a $g$-invariant subspace of $H^{1,1}(X,\bR)$ containing a K\"ahler current $B$.\footnote{A K\"ahler current $B$ is a real $(1,1)$-current such that $B - \epsilon \omega$ is a positive $(1,1)$-current for some $\epsilon > 0$.}
  Then $d_1(g)$ equals the spectral radius $\rho\big(g^*|_V\big)$.
  \item Suppose that $X$ is a smooth projective variety and $g$ is a surjective endomorphism of $X$.
  Then $\rho\big(g^*|_{H^{1,1}(X,\bR)}\big) = \rho\big(g^*|_{\NS_\bR(X)}\big)$.
\end{enumerate}
\end{lemma}

\begin{proof}
(1) It suffices to show that $d_1(g)\le \rho\big(g^*|_V\big)$.
Let $\cP$ be the closed cone in $H^{1,1}(X,\bR)$ consisting of classes of positive closed $(1,1)$-currents, and $\cC\coloneqq\cP\cap V$.
Note that $\cP$ is a strictly convex cone preserved by the pullback action $g^*$, so is $\cC$.
Replacing $V$ by the subspace spanned by $\cC$, we may assume that $V = \cC + (-\cC)$.
Take an interior point $B_1 \in \cC$.
Then $B'\coloneqq B_1+\eps B$ is still contained in the interior of $\cC$ (also in the interior of $\cP$) for sufficiently small $\eps>0$.
We can define a linear form $\chi \colon H^{1,1}(X,\bR)\to \bR$ by $\chi(\xi) = \int_X \xi \abxcup \omega^{n-1}$.
Note that for a non-trivial class $T$ in $\cP$, one has $\chi(T) > 0$ (cf.~\cite[Lemmas A.3 and A.4]{NZ09}).
So by applying \cite[Proposition A.2]{NZ09} to the triplets $\big(H^{1,1}(X,\bR), \cP, B'\big)$ and $(V, \cC, B')$, we obtain the following
$$d_1(g) = \lim_{m\to\infty} \chi\big((g^m)^*B'\big)^{\frac{1}{m}} = \rho\big(g^*|_V\big).$$
Note that in the proof above we have replaced $V$ by a subspace, so we actually prove that $d_1(g)\le \rho\big(g^*|_V\big)$.
This proves the assertion (1).

(2) In this case, $\NS_\bR(X)$ is a $g$-invariant subspace of $H^{1,1}(X,\bR)$ containing an ample divisor, whose first Chern class induces a K\"ahler class.
So the assertion (2) follows from the first one. This proves Lemma \ref{periodic-lemma-max-ev}.
\end{proof}

Consider the following hypotheses. We note that the natural map
$G|_{\NS_\bR(X)} \to G|_{\NS_\bC(X)}$ is an isomorphism, for the comparison
with the same hypothesis in \cite{Zhang16-TAMS}.

\begin{hyp} \label{periodic-Hyp-A}
Let $X$ be a normal projective variety of dimension $n\ge 2$,
and $G\le\Aut(X)$ such that the group $G^*\coloneqq G|_{\NS_\bC(X)}$ induced by the pullback action of $G$ on $\NS_\bC(X)$ is isomorphic to $\bZ^{\oplus n-1}$,
and every element of $G^*\setminus\{\id\}$ is of positive entropy.
\end{hyp}

\begin{hypp} \label{periodic-Hyp-A'}
Let $X$ be a normal projective variety of dimension $n\ge 2$,
and $G\le\Aut(X)$ such that $G|_{\NS_\bC(X)}$ is virtually solvable with maximal dynamical rank $r(G) = n-1$.
\end{hypp}

Obviously, \ref{periodic-Hyp-A} implies \ref{periodic-Hyp-A'}. The converse is also true up to finite-index by the following proposition.

\begin{prop} \label{periodic-prop-finite-null}
Suppose that $(X,G)$ satisfies {\rm \ref{periodic-Hyp-A'}}.
Then, replacing $G$ by a finite-index subgroup,
the null-entropy subset $N(G)$ of $G$ is a (necessarily normal) subgroup of $G$ and
virtually contained in the identity connected component $\Aut^0(X)$ of $\Aut(X)$, i.e.,
$$\big|N(G) : N(G)\cap \Aut^0(X)\big| <\infty.$$
In particular, the pair $(X,G)$ with $G$ replaced by a finite-index subgroup, satisfies {\rm \ref{periodic-Hyp-A}}.
\end{prop}

\begin{proof}
Let $\pi \colon \widetilde X\to X$ be an $\Aut(X)$-equivariant resolution of $X$ (cf.~Hironaka \cite{Hironaka77}).
Replacing $G$ by a finite-index subgroup, we may assume that
$G|_{\NS_\bC(\widetilde{X})}$ is solvable and has connected Zariski-closure in $\GL\big({\NS_\bC(\widetilde{X})}\big)$.
On the other hand, for any $g\in G$, we have $d_1(g|_{\widetilde X}) = d_1(g|_X)$ by Lemma \ref{periodic-lemma-dyn-degree}.
Thus, if we identify $G|_{\widetilde X}$ with $G|_X$, via the natural map $\pi$, then
$$N(G)|_{\widetilde X} = N(G)|_X = N(G|_X) = N(G|_{\widetilde X}),$$
where the second equality holds by definition.
By \cite[Theorem 4.1 (1)]{DHZ15}, we know that $N(G)|_{\widetilde X}$ is virtually contained in $\Aut^0(\widetilde X)$.
Hence $N(G)|_X$ is virtually contained in $\Aut^0(X)$,
since the $\Aut(X)$-equivariant birational morphism ${\widetilde X} \to X$ induces an isomorphism
$\Aut^0({\widetilde X}) \to \Aut^0(X)$.
Therefore, $N(G)|_{\NS_\bC(X)} = N(G)|_{\NS_\bC({\widetilde X})}$ is finite,
since the continuous part $\Aut^0({\widetilde X})$ acts trivially on the lattice $\NS({\widetilde X})$ (modulo torsion),
and hence acts trivially on $\NS_{\bC}({\widetilde X})$.
Now as in \cite[Lemma 3.1]{Zhang16-TAMS}, replacing $G$ by a finite-index subgroup,
we have $G|_{\NS_\bC(\widetilde{X})} \isom G|_{\widetilde X}\big/N(G|_{\widetilde X}) \isom \bZ^{\oplus n-1}$, and also
$G|_{\NS_\bC(X)} \isom \bZ^{\oplus n-1}$.
\end{proof}

Let $X$ be a normal projective variety of dimension $n\ge 2$, and $G\le\Aut(X)$.
Denote the union of all positive-dimensional $G$-periodic proper subvarieties of $X$ by $\Per_+(X,G)$, i.e.,
$$\Per_+(X,G)\coloneqq\bigcup_{Y \textrm{ is } G \textrm{-periodic}} Y,$$
where $Y$ runs over all positive-dimensional $G$-periodic proper subvarieties of $X$.

The result below follows from the equivariance assumption.

\begin{lemma} \label{periodic-lemma-periodic}
Let $f \colon X_1\to X_2$ be a $G$-equivariant generically finite surjective morphism.
Then we have the following relation:
$$\Per_+(X_1,G) = f^{-1}\big(\Per_+(X_2,G)\big),$$
where $f^{-1}$ denotes the set-theoretical inverse.
\qed
\end{lemma}

In the rest of this section, we prove some preliminary results under \ref{periodic-Hyp-A}.
First note that if $X$ is smooth, a {\it quasi-nef sequence} with $1 \le k \le n$ \index{quasi-nef sequence}
$$0\ne L_1 \cdots L_k \, \in \, \overline{L_1 \cdots L_{k-1} \cdot \Nef(X)} \, \subseteq \, H^{k, k}(X, \bR)$$
was constructed in \cite[\S2.7]{Zhang09-Invent}.
Here as in \cite[Lemma 3.4]{Zhang16-TAMS}, we give a generalization of \cite[Theorem 4.3]{DS04} to the singular case.
Besides, we introduce a nef and big $\bR$-Cartier divisor $A$,
which plays an important role in running the Log Minimal Model Program (LMMP for short) with scaling
(cf.~\cite[Corollary 1.4.2]{BCHM10} or \cite[Theorem 1.9 (i)]{Birkar12}).

\begin{lemma} \label{periodic-lemma-big-A}
Suppose that $(X,G)$ satisfies {\rm \ref{periodic-Hyp-A}}.
Then there are nef $\bR$-Cartier divisors $L_i$ for $1\le i\le n$ with $L_1 \cdots L_n\ne 0$, such that for any $g\in G$,
$$g^*L_i \num \exp\chi_i(g)L_i \text{ (numerical equivalence)}$$
for some characters $\chi_i \colon G\to (\bR, +)$, and the group homomorphism
$$\varphi \colon G \to \big(\bR^{\oplus n-1}, +\big), \quad  g \mapsto \big(\chi_1(g), \dots, \chi_{n-1}(g)\big)$$
has image a spanning (discrete) lattice of $\big(\bR^{\oplus n-1}, +\big)$ and
satisfies the following:
\begin{equation}\label{periodic-eqn-iso} \tag{\dag}
\Ker\varphi = N(G), \quad  G^*\isom G/N(G) \xrightarrow{\sim} \im\varphi \isom \bZ^{\oplus n-1}.
\end{equation}
In particular,
$$A\coloneqq\sum_{i=1}^n L_i$$
is a nef and big $\bR$-Cartier divisor.
\end{lemma}

\begin{proof}
Let $\pi \colon \widetilde X\to X$ be a $G$-equivariant resolution of $X$ due to Hironaka \cite{Hironaka77}.
We follow the proof of \cite[Theorem 4.3]{DS04},
and consider the action of $G$ on the pullback $\pi^*\Nef(X)$ of the nef cone $\Nef(X)\subset \NS_\bR(X)$
(instead of the K\"ahler cone $\cK(\widetilde X)\subset H^{1,1}(\widetilde X,\bR)$ there).
Then there are nef $\bR$-Cartier divisors $\pi^*L_i$ with $1\le i\le n$ on $\widetilde X$ as common eigenvectors of $G$ acting on $\pi^*\NS_\bR(X)$, i.e., $g^*(\pi^*L_i) \num \exp\chi_i(g)\pi^*L_i$,
such that $\chi_1 + \cdots + \chi_n = 0$ and the induced homomorphism $\varphi$ satisfies \eqref{periodic-eqn-iso}.
By taking a push-forward, these $L_i$ satisfy $g^*L_i \num \exp\chi_i(g)L_i$.
For details, see \cite[Lemma 3.4]{Zhang16-TAMS} or \cite[proof of Theorems 1.2 and 2.2, p.~137]{Zhang13}.

Note that $A$ is nef by its definition.
Then it is big because $$A^n=(L_1+ \cdots +L_n)^n\ge L_1 \cdots L_n>0.$$
The latter inequality follows from \cite[Lemma 4.4]{DS04}.
More precisely, that lemma implies that $L_1 \cdots L_n$ is nonzero and hence positive since these $L_i$ are nef.
\end{proof}

For a nef $\bR$-Cartier divisor $L$ on a projective variety $X$, define the {\it null locus} of $L$ as
\index{null locus of a nef divisor}
$$\Null(L) \coloneqq \bigcup_{L|_Z \textrm{ is not big}}Z,$$
where $Z$ runs over all positive-dimensional proper subvarieties of $X$.
Note that $L|_Z$ is nef, so it is not big if and only if $L^{\dim Z}\cdot Z=0$.

\begin{lemma}[{cf.~\cite[Lemma 3.9]{Zhang16-TAMS}}] \label{periodic-lemma-Null-A}
Suppose that $(X,G)$ satisfies {\rm \ref{periodic-Hyp-A}}.
Then $$\Per_+(X,G) = \Null(A),$$
and it is a Zariski closed proper subset of $X$, where $A$ is constructed in Lemma \ref{periodic-lemma-big-A}.
In particular, $A$ is ample if and only if every $G$-periodic proper subvariety of $X$ is a point.
\end{lemma}

Below is the key proposition in \cite{Zhang16-TAMS} which was used to prove \cite[Theorem 2.4]{Zhang16-TAMS}.
Note that we do {\it not} need the pseudo-effectivity of $K_X + D$ or $\dim X\ge 3$.

\begin{prop}[{cf.~\cite[Proposition 3.11]{Zhang16-TAMS}}] \label{periodic-prop-Zhang-klt}
Suppose that $(X,G)$ satisfies {\rm \ref{periodic-Hyp-A}}.
Assume that for some effective $\bR$-divisor $D$ whose irreducible components are $G$-periodic,
the pair $(X,D)$ has at worst $\bQ$-factorial klt singularities.
Let $A = \sum L_i$ be the nef and big $\bR$-Cartier divisor as in Lemma \ref{periodic-lemma-big-A}.
Then after replacing $G$ by a finite-index subgroup and $A$ by a large multiple, the following assertions hold.
\begin{enumerate}[{\em (1)}]
  \item There is a sequence $\tau_s \circ \cdots \circ \tau_0$ of $G$-equivariant birational maps:
  \begin{equation*}\label{periodic-eqn-seq1} \tag{$\dag$}
  X=X_0 \overset{\tau_0}{\ratmap} X_1 \overset{\tau_1}{\ratmap} \cdots \overset{\tau_{s-1}}{\ratmap} X_s \overset{\tau_s}{\longrightarrow} X_{s+1}=Y
  \end{equation*}
  such that each $\tau_j \colon X_j \ratmap X_{j+1}$ for $0\le j < s$ is either a divisorial contraction of a $(K_{X_j} + D_j)$-negative extremal ray or a $(K_{X_j} + D_j)$-flip;
  the $\tau_s \colon X_s\to X_{s+1}=Y$ is a birational morphism such that $K_{X_s}+D_s=\tau_s^*(K_Y+D_Y)$ is $\bR$-Cartier;
  here $D_i\subset X_i$ for $0\le i\le s+1$ is the push-forward of $D$ and $D_Y\coloneqq D_{s+1}$.
  \item For $0\le i\le s+1$, the push-forward $A_i$ of $A$ on $X_i$ is a nef and big $\bR$-Cartier divisor.
  \item For $0\le i\le s+1$, the pair $(X_i,D_i + A_i)$ and hence the pair $(X_i,D_i)$ have at worst klt singularities;
  $X_j$ is $\bQ$-factorial for $0\le j\le s$.
  \item $K_Y + D_Y + A_Y$ is an ample $\bR$-Cartier divisor, where $A_Y\coloneqq A_{s+1}$.
  \item For $0\le i\le s+1$, the union of all positive-dimensional $G$-periodic proper subvarieties of each $X_i$ is a Zariski closed proper subset of $X_i$.
  Further, $A_i|_Z \num 0$ for every positive-dimensional $G$-periodic proper subvariety $Z$ of $X_i$.
  \item For $0\le i\le s+1$, the induced action of $G$ on each $X_i$ is biregular.
  Further, each $(X_i,G)$ also satisfies {\rm \ref{periodic-Hyp-A}}.
\end{enumerate}
\end{prop}

Note that if $(X,D)$ is only a dlt pair, one has the following proposition (but need $K_X+D$ to be pseudo-effective).
The main idea is to apply Proposition \ref{periodic-prop-Zhang-klt} to the klt pair $\big(X,(1-\eps)D\big)$ for some $0<\eps\ll 1$.

\begin{prop}[{cf.~\cite[Proposition 2.6]{Zhang16-TAMS}}] \label{periodic-prop-Zhang-dlt}
Suppose that $(X,G)$ satisfies {\rm \ref{periodic-Hyp-A}}.
Suppose further that for some effective $\bQ$-divisor $D$ whose irreducible components are $G$-periodic,
the pair $(X,D)$ has at worst $\bQ$-factorial dlt singularities, and $K_X+D$ is a pseudo-effective divisor.
Then, after replacing $G$ by a finite-index subgroup, the following assertions hold.
\begin{enumerate}[{\em (1)}]
  \item There is a $G$-equivariant birational map $X\ratmap Y$ which is surjective in codimension-$1$.
  Moreover, the induced action of $G$ on $Y$ is biregular.
  \item The pair $(Y,D_Y)$ has only log canonical singularities and $K_Y + D_Y \sim_\bQ 0$, where $D_Y$ is the push-forward of $D$.
  \item Every $G$-periodic positive-dimensional proper subvariety of $Y$ is contained in the support of $D_Y$.
\end{enumerate}
\end{prop}

Under \ref{periodic-Hyp-A}, the rank of the N\'eron--Severi group has the following lower bound (see also \cite[Theorem 4.3]{DS04}).

\begin{lemma} \label{periodic-lemma-rho}
Suppose that $(X,G)$ satisfies {\rm \ref{periodic-Hyp-A}}. Then we have:
\begin{enumerate}[{\em (1)}]
  \item The Picard number $\rho(X) \ge n$.
  \item Assume there exists a numerically non-zero $\bR$-Cartier divisor $M$ such that $g^*M\num M$ for any $g\in G$.
  Then $\rho(X)\ge n+1$.
  \item If $\rho(X) = n$ and $K_X$ is $\bQ$-Cartier, then $K_X\num 0$.
\end{enumerate}
\end{lemma}

\begin{proof}
(1) We use the notations as in Lemma \ref{periodic-lemma-big-A}.
Namely, we have $n$ distinct characters $\chi_i$ whose corresponding common eigenvectors are nef $\bR$-Cartier divisors $L_i$, respectively. It then follows that these $L_i$'s are linearly independent in $\NS_\bR(X)$, so $\rho(X)\ge n$.

(2) The assumption is equivalent to say that the $\bR$-Cartier divisor $M$ is a non-zero common eigenvector of $G$ corresponding to the trivial character (i.e., $G\mapsto 0$).
Then the same reason as in the assertion (1) implies that $M$ and all $L_i$'s are linearly independent in $\NS_\bR(X)$.

(3) It follows from the assertion (2) by taking $M = K_X$.
\end{proof}

\begin{prop} \label{periodic-prop-PropB}
Suppose that $(X,G)$ satisfies {\rm \ref{periodic-Hyp-A}} and $X$ has at worst $\bQ$-factorial klt singularities.
Let $B_1, \dots, B_s$ be distinct $G$-periodic prime divisors on $X$.
Then we have:
\begin{enumerate}[{\em (1)}]
  \item If the irregularity $q(X) = 0$, then $B_1, \dots, B_s$ are linearly independent in $\NS_{\bQ}(X)$.
  In particular, they are linearly independent in $\NS_\bR(X)$.
  \item If there is a projective birational morphism $X\to X'$ such that $B_1, \dots, B_s$ are exceptional divisors, then they are linearly independent in $\NS_{\bR}(X)$.\footnote{The linear independence of exceptional divisors is a purely birational geometric property. Actually, we do not need $B_i$ to be $G$-periodic in the proof of the assertion (2).}
  \item If $B_1, \dots, B_s$ are linearly independent in $\NS_{\bR}(X)$, then $s \le \rho(X) - n$.
\end{enumerate}
\end{prop}

\begin{proof}
(1) Replacing $G$ by a finite-index subgroup, we may assume that all of $B_i$ have been stabilized by $G$.
Suppose to the contrary that these $B_i$ are linearly dependent in $\NS_{\bQ}(X)$.
Then we have $\sum_{i=1}^s a_iB_i \num 0$ in $\NS_\bQ(X)$ for some $a_i\in \bQ$, not all zero.
After rearranging the order of $B_i$,
we may assume that $E_1\coloneqq \sum_{i=1}^{s_1} a_iB_i \num \sum_{j=s_1+1}^{s_2} b_jB_j \eqqcolon E_2$,
where $a_i, b_j = - a_j$ are positive rational numbers.
Since $q(X)=0$ by assumption, we have $E_1\sim E_2$ (linear equivalence) after replacing $E_i$ by some multiples.
Hence the Iitaka $D$-dimension $\kappa \coloneqq \kappa(X, E_1)\ge 1$.

Replacing $E_1$ by some $mE_1$, we may assume that the map $\Phi_{|E_1|} \colon X\ratmap \bP H^0\big(X,\cO_X(E_1)\big)$ gives rise to
the Iitaka fibration associated to $E_1$, so that its image has dimension equal to $\kappa$.
Take a $G$-equivariant resolution $\pi \colon \widetilde X\to X$ (cf.~Hironaka \cite{Hironaka77}),
such that the linear system $|\pi^*E_1|$ equals $|M|+F$,
where $M$ is base point free, $F$ is the fixed component of $|\pi^*E_1|$, and both of their divisor classes are $G$-stable.
Now the rational map $\Phi_{|E_1|} \colon X\ratmap \bP H^0\big(X,\cO_X(E_1)\big)$ is birational to the $G$-equivariant morphism
$\Phi_{|M|} \colon \widetilde X\to Y\subset \bP H^0\big(\widetilde X,\cO_{\widetilde X}(M)\big)$ with $\dim Y=\kappa$.

If $\kappa=n$, then $M$ is a nef and big divisor.
So by \cite[Lemma 2.23]{Zhang09-JDG}, $G$ is virtually contained in $\Aut^0(\widetilde X)$ and hence is of null entropy on $\widetilde X$,
and also on $X$ (cf.~Lemma \ref{periodic-lemma-dyn-degree}).
This contradicts that the dynamical rank $r(G) = n-1 \ge 1$.
Thus we have $1\le \kappa \le n-1$.
In other words, $\Phi_{|M|}$ is a non-trivial $G$-equivariant fibration with general fibres of dimension $n-\kappa\in \{1, \dots, n-1 \}$.
Then by \cite[Lemma 2.10]{Zhang09-Invent}, the dynamical rank $r(G) \le n-2$, which contradicts \ref{periodic-Hyp-A}.
So we have proved the linear independence of these $B_i$ in $\NS_\bQ(X)$.

The second part of the assertion (1) follows from a linear algebra argument.

(2) Suppose that these $B_i$ are linearly dependent in $\NS_{\bR}(X)$.
Then we have $\sum_{i=1}^s a_iB_i \num 0$ in $\NS_\bR(X)$ for some $a_i\in \bR$, not all zero.
As in the proof of the first assertion, we may assume that $E_1\coloneqq \sum_{i=1}^{s_1} a_iB_i \num \sum_{j=s_1+1}^{s_2} b_jB_j \eqqcolon E_2$,
where $a_i, b_j$ are positive real numbers.
By the negativity of contraction (cf.~\cite[Lemma 3.6.2 (1)]{BCHM10}),
there exists a $B_{i_0}$ for some $1\le i_0\le s_1$ which is covered by curves $\Sigma$ such that $E_1 \cdot \Sigma < 0$.
However, for a general curve $\Sigma$ in the covering family of $B_{i_0}$,
we have $B_j \cdot \Sigma \ge 0$ for any $s_1+1\le j\le s_2$ and hence $E_2 \cdot \Sigma \ge 0$.
This is a contradiction.

(3) We continue using the notations as in Lemmas \ref{periodic-lemma-big-A} and \ref{periodic-lemma-rho}.
By the argument similar to the proof of Lemma \ref{periodic-lemma-rho} (2),
we can show that $L_1,\dots,L_n,B_1,\dots,B_s$ are linearly independent in $\NS_\bR(X)$.
Thus we have $n+s\le \rho(X)$.
This ends the proof of Proposition \ref{periodic-prop-PropB}.
\end{proof}

The following lemma generalizes a fact, which asserts that every effective divisor on an abelian variety is nef.

\begin{lemma} \label{periodic-lemma-eff-on-torus}
Suppose that $\pi \colon T\to X$ is a finite surjective morphism between normal projective varieties.
Suppose further that $T$ satisfies one of the following conditions.
\begin{enumerate}[{\em (i)}]
\item $T$ has at worst klt singularities and contains no rational curve; $K_T\sim_\bQ 0$.
\item $T$ is an abelian variety.
\end{enumerate}
Then we have:
\begin{enumerate}[{\em (1)}]
  \item Every pseudo-effective $\bR$-Cartier divisor on $X$ is nef.
  \item Every big $\bR$-Cartier divisor on $X$ is ample.
\end{enumerate}
\end{lemma}

\begin{proof}
Since $\pi$ is finite and by the projection formula,
an $\bR$-Cartier divisor $D$ on $X$ is pseudo-effective, big, nef or ample if and only if so is $\pi^*D$.
Thus we only need to prove this lemma for $X=T$.
Further, we may assume that $T$ satisfies the condition (i) since the condition (ii) implies the condition (i).
By the Kodaira lemma, which states that every big $\bR$-divisor is the sum of an ample $\bQ$-divisor and an effective $\bR$-divisor
(cf.~\cite[Lemma 3.16]{Nakayama04}), it suffices to prove the assertion (1).
Since the cone of all pseudo-effective $\bR$-Cartier divisors on $T$ is
the closure of the cone of all effective $\bR$-Cartier divisors on $T$ in $\NS_\bR(T)$ and the nef cone $\Nef(T)$ is closed,
we only need to show that every effective $\bR$-Cartier divisor on $T$ is nef.
For this, it suffices to show that every effective Cartier divisor on $T$ is nef.
Suppose to the contrary that some effective Cartier divisor $D$ on $T$ is not nef.
By \cite[Corollary 2.35]{KM98}, $(T,\eps D)$ is klt for all sufficiently small rational number $\eps>0$.
Now $K_T + \eps D \sim_\bQ \eps D$ is not nef.
Therefore, applying the cone theorem in MMP to $(T,\eps D)$ (cf.~\cite[Theorem 3.7]{KM98}),
we obtain an extremal rational curve on $T$, which contradicts the condition (i).
This proves Lemma \ref{periodic-lemma-eff-on-torus}.
\end{proof}

The following result proves the implication (1) $\Longrightarrow$ (2) in Theorem \ref{periodic-thm-ThmC}.

\begin{lemma} \label{periodic-lemma-A}
Let $X$ be a quasi-\'etale torus quotient $T/F$ for some abelian variety $T$ and a finite group $F$
acting freely outside a codimension-$2$ subset of $T$, and $G\le\Aut(X)$ such that $(X,G)$ satisfies {\rm \ref{periodic-Hyp-A}}.
Then $X$ has no positive-dimensional $G$-periodic proper subvariety.
\end{lemma}

\begin{proof}
Let $\widetilde T\to X$ be the Galois covering (or minimal split covering in the sense of Beauville; see \cite[\S3]{Beauville83})
corresponding to the unique maximal lattice $L$ in $\pi_1\big(X\setminus \Sing X\big)$ such that $\widetilde T$ is an abelian variety.
Then there exists a group $\widetilde{G}$ (which is the lifting of $G$) acting faithfully on $\widetilde T$, such that $G=\widetilde G/F$.
See also \cite[\S2.15]{Zhang13}.
Note that the action of $G$ on $X$ can be identified with a not necessarily faithful action of $\widetilde{G}$ on $X$ (with finite kernel).
Replacing $\widetilde{G}$ by a finite-index subgroup,
we may assume that the new $\widetilde G$ acts faithfully on both $\widetilde{T}$ and $X$ (cf.~\cite[Lemma 2.4]{Zhang13}),
and both $(\widetilde{T}, \widetilde G)$ and $(X, \widetilde G)$ satisfy \ref{periodic-Hyp-A} (cf.~\cite[Lemma 3.1]{Zhang16-TAMS}).
By Lemma \ref{periodic-lemma-eff-on-torus}, the nef and big $\bR$-Cartier divisor $\widetilde A$ on $\widetilde{T}$ as constructed in Lemma \ref{periodic-lemma-big-A}, is ample.
Hence every $\widetilde G$-periodic proper subvariety of $\widetilde T$ is a point (see Lemma \ref{periodic-lemma-Null-A}).
The same holds for $X$ by Lemma \ref{periodic-lemma-periodic}.
\end{proof}

\section{Some general results from birational geometry}\label{periodic-section-pre2}

In this section, we establish some general results which will be used in the last two sections to prove our main theorems and propositions.
They should be of interest in their own right.

We first quote the following result, which will be frequently used in the sequel of the paper.

\begin{lemma}[{cf.~\cite[Corollary 1.5]{HM07}}] \label{periodic-lemma-HM-RCC}
Let $(X,\Delta)$ be a dlt pair for some effective $\bQ$-divisor $\Delta$
and
$\phi \colon W\to X$ a birational projective morphism.
Denote by $\Exc (\phi)$ the exceptional locus of $\phi$, \index{exceptional locus of a rational map}
i.e., the set of points on $W$ at which $\phi$ is not an isomorphism.
Then we have:
\begin{enumerate}[{\em (1)}]
  \item Every fibre of $\phi$ is rationally chain connected.
  \item Every irreducible component of $\Exc (\phi)$ is uniruled.
  In particular, if $D$ is a non-uniruled prime divisor on $W$, then so is the push-forward of $D$ on $X$.
\end{enumerate}
\end{lemma}

The proposition below gives an affirmative answer to Question \hyperref[periodic-qestion-QnA]{\ref{periodic-qestion-QnA} (2)}.

\begin{prop} \label{periodic-prop-PropA}
Suppose that $(X,G)$ satisfies {\rm \ref{periodic-Hyp-A}}.
Suppose further that $X$ is $G$-equivariant birational to a quasi-\'etale torus quotient. Then we have:
\begin{enumerate}[{\em (1)}]
  \item Every connected component $Z_k$ of $\Per_+(X,G)$ (i.e., the union of all positive-dimensional $G$-periodic proper subvarieties of $X$) is rationally chain connected.
  \item Every irreducible component of $\Per_+(X,G)$ is uniruled. In particular, Question \hyperref[periodic-qestion-QnA]{\ref{periodic-qestion-QnA} (2)} has a positive answer.
\end{enumerate}
\end{prop}

\begin{proof}
Since the assertion (2) follows readily from the first one, we prove only the assertion (1).
Suppose that $X$ is $G$-equivariant birational to a quasi-\'etale torus quotient $Y\coloneqq T/F$ for some abelian variety $T$ and a finite group $F$ (note that $Y$ is klt).
Since the image of a rationally chain connected Zariski closed set is still rationally chain connected,
we may replace $X\ratmap Y$ by a $G$-equivariant resolution of indeterminacy (cf.~Hironaka \cite{Hironaka77}),
and assume that $X\to Y$ is already a $G$-equivariant birational morphism
(see Lemma \ref{periodic-lemma-periodic} and \cite[Lemma 3.1]{Zhang16-TAMS}).
Note that the image of $Z_k$ on $Y$ is $G$-periodic and hence a point $P$ by Lemma \ref{periodic-lemma-A}.
By Zariski's main theorem, the inverse image on $X$ of the point $P$ on the normal variety $Y$ is connected.
This inverse of $P$ is also $G$-periodic and contains $Z_k$, so it equals $Z_k$, since $Z_k$ is a connected component of $\Per_+(X,G)$.
Then by Lemma \ref{periodic-lemma-HM-RCC}, $Z_k$ is rationally chain connected.
\end{proof}

Below is an easy fact whose proof is left to the reader.

\begin{lemma}[{cf.~\cite[Exercise 2.1]{HK10}}] \label{periodic-lemma-Q-Cartier}
Let $X$ be a normal projective variety and $D$ a Weil $\bQ$-divisor.
If $D$ is $\bR$-Cartier, then it is $\bQ$-Cartier.
\end{lemma}

It is well-known that the birational automorphism group of a projective variety of general type is finite (cf.~\cite[Theorem 14.10]{Ueno75}).
Below is a similar result.

\begin{lemma} \label{periodic-lemma-finite-auto}
Let $X$ be a non-uniruled normal projective variety, and $G\le \Aut(X)$ such that the linear equivalence class of an ample divisor $H$ is $G$-periodic.
Then $G$ is finite.
\end{lemma}

\begin{proof}
Replacing $H$ by a large multiple, we may assume that $H$ is very ample and hence
the complete linear system $|H|$ defines a closed embedding from $X$ into some projective space $\bP H^0\big(X,\cO_X(H)\big)\isom \bP^N$.
Identify $X$ with its image.
Replacing $G$ by a finite-index subgroup, we may assume that $G$ itself stabilizes the linear equivalence class of $H$.
Thus the above embedding is $G$-equivariant.
So $G$ is contained in $\Aut(\bP^N, X)$, the Zariski closed subgroup of $\Aut(\bP^N)$ stabilizing $X$.
If $G$ were infinite, then the linear algebraic group $\Aut(\bP^N, X)$ contains the $1$-dimensional linear algebraic group $\bG_a$ or $\bG_m$,
whose orbit of a general point is a rational curve.
But our $X$ is non-uniruled. This is a contradiction.
Hence $G$ is finite.
\end{proof}

We give a criterion for the log canonical divisor $K_X+D$ to be pseudo-effective.
See \cite[Theorem 1.4 or 3.7]{LZ15} for a more general form.

\begin{lemma} \label{periodic-lemma-ps}
Let $X$ be a rationally connected normal projective variety,
and $D$ a non-uniruled prime divisor such that $K_X+D$ is $\bQ$-Cartier.
Then $K_X+D$ is pseudo-effective.
\end{lemma}

\begin{proof}
Take a log resolution $\widetilde X\to X$ for the pair $(X, D)$, and denote by $\widetilde D$ the proper transform of $D$.
Note that the push-forward of a pseudo-effective divisor is still pseudo-effective.
Hence we may replace the pair $(X,D)$ by $(\widetilde X,\widetilde D)$, and assume that it is $\bQ$-factorial dlt now.

Suppose to the contrary that $K_X+D$ is not pseudo-effective.
We shall follow the proof of \cite[Theorem 3.7]{LZ15}.
After running a $(K_X+D)$-MMP with an ample scaling, we reach a Fano fibration $h \colon W\to Y$ as follows (cf.~\cite[Corollary 1.3.3]{BCHM10})
\[\xymatrix@C=.8cm{
  X=X_0 \ar@{-->}[r]^(.6){f_0} & X_1 \ar@{-->}[r]^{f_1} & \cdots \ar@{-->}[r]^(.4){f_{m-2}} & X_{m-1} \ar@{-->}[r]^(.4){f_{m-1}} & X_m \eqqcolon W \ar[d]^{h}\\
            & & & & Y.}
\]
Note that each $f_i$ above is either a divisorial contraction of a $(K_{X_i}+D_i)$-negative extremal ray
or a $(K_{X_i}+D_i)$-flip, where $D_i\subset X_i$ is the push-forward of $D$.
So $(X_i,D_i)$ is still $\bQ$-factorial and dlt (cf.~\cite[Corollary 3.44]{KM98}).
Thus $D_W \coloneqq D_m$, as the push-forward of $D$ on $W$, is still a non-uniruled prime divisor since so is $D$
(see Lemma \ref{periodic-lemma-HM-RCC}).
Then the argument in \cite[4.7. Proof of Theorem 3.7, the second paragraph]{LZ15} asserts that $h \colon W\to Y$ is a $\bP^1$-fibration with $D_W$ a cross-section.
Hence $D_W$ is birational to $Y$ via the restriction map $h|_{D_W}$.
Since $X$ is rationally connected, so are each $X_i$ and the $h$-image $Y$ of $W$.
Thus $D_W$ is rationally connected and hence uniruled.
This is a contradiction. So the lemma is proved.
\end{proof}

The following lemma provides sufficient conditions to have canonical singularities.
Note that under the condition (ii) it is a log-version of \cite[Lemma 2.4]{HMZ14}.

\begin{lemma} \label{lemma-can-sing}
Let $X$ be a normal projective variety of dimension $n$, and $D$ an effective Weil $\bQ$-divisor such that $K_X+D$ is $\bQ$-Cartier and $K_X+D \num 0$ (numerical equivalence).
Suppose that one of the following two conditions hold.
\begin{enumerate}[{\em (i)}]
  \item $X$ is rationally connected and $D$ is a non-uniruled prime divisor.
  \item $X$ is non-uniruled.
\end{enumerate}
Then $(X, D)$ has at worst canonical (and hence dlt) singularities.
Moreover, under the condition (i), the prime divisor $D$ itself as a variety is normal;
under the condition (ii), we further have $D = 0$ and hence $X$ has at worst canonical singularities.
\end{lemma}

\begin{proof}
Take a log resolution $\pi \colon \widetilde X\to X$ for the pair $(X, D)$ and denote the proper transform of $D$ by $\widetilde D$.
Under the condition (i), $\widetilde X$ is still rationally connected and $\widetilde D$ is non-uniruled.
So it follows from Lemma \ref{periodic-lemma-ps} that $K_{\widetilde X}+\widetilde D$ is pseudo-effective.
Under the condition (ii), $\widetilde X$ is also non-uniruled and hence $K_{\widetilde X}$ is pseudo-effective by \cite[Theorem 2.6]{BDPP13}.
Without loss of generality, we may assume that $K_{\widetilde X}+\widetilde D$ admits a Zariski $\sigma$-decomposition $K_{\widetilde X} + \widetilde D = P + N$,
where the $\bR$-divisors $P$ and $N$ are the movable part and the negative part of $K_{\widetilde X} + \widetilde D$,
respectively (cf.~\cite[Ch.~III, \S1.b]{Nakayama04}).
On the other hand, we have
$$K_{\widetilde X} + \widetilde D = \pi^*(K_X + D) + E_1 - E_2 \num E_1 - E_2, $$
where $E_1$ and $E_2$ are effective $\pi$-exceptional divisors and have no common component.
Now the same argument as in \cite[Lemma 2.4]{HMZ14} eventually shows that $E_2 = 0$,
and hence $(X,D)$ has only canonical singularities by definition.
Note however that to conclude $D = 0$ under the condition (ii) we need to consider the Zariski $\sigma$-decomposition of the pseudo-effective divisor $K_{\widetilde X}$.
The rest of the proof is similar to \cite[Lemma 2.4]{HMZ14} and left to the reader.
\end{proof}

When $X$ is a surface, we have a more specific description of $X$ and its periodic curves.

\begin{lemma} \label{periodic-lemma-surface}
Let $X$ be a normal projective surface with an automorphism $g$ of positive entropy, and $C$ a $g$-periodic curve.
Then either $X$ is a rational surface, or $C$ is a rational curve.
\end{lemma}

\begin{proof}
Replacing $X$ by a $g$-equivariant resolution of singularities due to Hironaka \cite{Hironaka77},
we may assume that $X$ is smooth.
Since $X$ admits an automorphism of positive entropy, by \cite[Proposition 1]{Cantat99},
either $X$ is a rational surface, or it has Kodaira dimension $\kappa(X) = 0$.

Thus we have only to consider (and rule out) the case where $\kappa(X) = 0$ and $C$ is irrational.
Let $X \to X_m$ be the smooth blowdown to the (unique smooth) minimal model of $X$.
Note that the image $C_m$ of $C$ is still a curve by Lemma \ref{periodic-lemma-HM-RCC},
and $g$ descends to an automorphism on $X_m$.
So we may replace $(X,C)$ by $(X_m,C_m)$, and assume that $X$ is minimal.
Hence $K_X \sim_{\bQ} 0$. More precisely, $X$ is either a $K3$ surface, or an Enriques surface, or an abelian surface (cf.~\cite[Proposition 1]{Cantat99}).

Replacing $g$ by some power, we may assume that $g$ stabilizes the curve $C$.
The generalized Perron--Frobenius theorem due to Birkhoff asserts that
$(g^{\pm 1})^*L_{g^{\pm 1}} \num d_1(g^{\pm 1}) L_{g^{\pm1 }}$ for some nonzero nef divisors $L_{g^{\pm 1}}$.
Then $A\coloneqq L_g+L_{g^{-1}}$ is nef and also big since $A^2 \ge L_g \cdot L_{g^{-1}} > 0$.
It is perpendicular to $C$ because $d_1(g^{\pm 1}) > 1$.
Indeed,
\begin{equation}\label{periodic-eqn-perp}
L_{g^{\pm 1}}\cdot C = (g^{\pm 1})^*(L_{g^{\pm 1}}\cdot C) = (g^{\pm 1})^*L_{g^{\pm 1}}\cdot (g^{\pm 1})^*C = d_1(g^{\pm 1})L_{g^{\pm 1}}\cdot C.
\end{equation}
It follows that $L_{g^{\pm 1}}\cdot C=0$, and hence $A\cdot C=0$.
Thus $C^2<0$ by the Hodge index theorem, since $A^2 > 0$.

On the other hand, by the arithmetic genus formula, we have
$$0 > C^2=(K_X+C)\cdot C=2p_a(C)-2\ge 0,$$
since $C$ is irrational.
This is a contradiction. Lemma \ref{periodic-lemma-surface} is proved.
\end{proof}

\begin{remark} \label{periodic-remark-surface}
Suppose $X$ is a smooth projective rational surface with an automorphism $g$ of positive entropy. Then $K_X^2<0$.
Indeed, since $g^*K_X\sim K_X$, we have $A \cdot K_X = 0$ as calculated in the equation \eqref{periodic-eqn-perp} of the lemma above with $C$ replaced by $K_X$.
Hence either $K_X\num 0$, or $K_X^2<0$.
Since $X$ is a smooth rational surface, $K_X$ is not numerically trivial, so $K_X^2 < 0$.

If $C$ is a $g$-periodic curve on $X$, then the arithmetic genus $p_a(C) \le 1$.
Otherwise, the Riemann--Roch theorem and the Serre duality imply that
$$h^0\big(X, \cO_X(K_X+C)\big) \ge \chi(\cO_X) + \frac{1}{2} \, C\cdot (K_X+C) = p_a(C) \ge 2.$$
So the nef part of the Zariski-decomposition of $K_X + C$ is nonzero and $g$-invariant, contradicting $d_1(g) > 1$.
\end{remark}

We end this section with the following rigidity result for the proof of Proposition \ref{periodic-prop-finite-periodic} (\ref{periodic-prop-finite-periodic_2}).
It follows from \cite[Lemma 1.6]{KM98} and \cite[Proposition 1.14 or Lemma 1.15]{Debarre01}.

\begin{lemma}[Rigidity Lemma] \label{periodic-lemma-rigidity}
Let $f \colon X\to Y$ be a projective surjective morphism of normal varieties.
Suppose that all fibres of $f$ are connected and of the same dimension.
Let $f' \colon X\to Y'$ be another projective morphism of varieties such that
$f'(f^{-1}(y_0))$ is a point for some $y_0\in Y$.
Then there is a unique morphism $\pi \colon Y\to Y'$ such that $f'=\pi\circ f$.
\end{lemma}

\section{Proofs of Theorems \ref{periodic-thm-ThmA'} and \ref{periodic-thm-ThmB'}}\label{periodic-section-proofs}

Our proof of Theorem \ref{periodic-thm-ThmA'} will rely on the following two lemmas.

\begin{lemma} \label{periodic-lemma-Null}
Suppose that we have the sequence \eqref{periodic-eqn-seq1} of $G$-equivariant birational maps as in Proposition \ref{periodic-prop-Zhang-klt}.
Then we have the following relations among the $\Per_+(X_i,G)$.
\begin{enumerate}[{\em (1)}]
  \item For a divisorial contraction $\tau_i$ with $0\le i< s$ and
  for the birational morphism $\tau_i$ with $i = s$, we have
  $$\Per_+(X_i,G) = \tau_i^{-1} \big(\Per_+(X_{i+1},G)\big).$$
  Moreover, every irreducible component of the exceptional locus $\Exc (\tau_i)$ is uniruled.
  \item If $\tau_i$ is a $(K_{X_i} + D_i)$-flip for some $0\le i< s$:
  \[\xymatrix{
  X_i \ar@{-->}[rr]^(0.4){\tau_i} \ar[dr]_{f}
                &  &    X_{i+1}=X_i^+ \ar[dl]^{f^+}    \\
                & V_i },
  \]
  then there is a Zariski closed subset $\Delta_i \subset V_i$ such that
  the flipping locus $\Exc (f) = f^{-1}(\Delta_i)$ and the flipped locus $\Exc (f^+) = (f^+)^{-1}(\Delta_i).$
  Further, $\Per_+(X_i, G) = f^{-1} \big(\Per_+(V_i, G)\big)$ and $\Per_+(X_{i+1}, G) = (f^+)^{-1} \big(\Per_+(V_i, G)\big).$
  Every irreducible component of the flipping locus $\Exc (f)$ or the flipped locus $\Exc (f^+)$ is uniruled.
\end{enumerate}
\end{lemma}

\begin{proof}
(1) The first part follows directly from Lemma \ref{periodic-lemma-periodic}.
For the second one, we know that every $(X_i, D_i)$ is klt (so is dlt) by Proposition \ref{periodic-prop-Zhang-klt} (3).
Then it follows from Lemma \ref{periodic-lemma-HM-RCC} that every irreducible component of $\Exc (\tau_i)$ is uniruled.

(2) The first part follows from the uniqueness of flips (cf.~\cite[Lemma 6.2 and Corollary 6.4]{KM98}).
Now the second part follows, using also the $G$-equivariance of the morphisms $f$ and $f^+$ and Lemma \ref{periodic-lemma-periodic}.

Hence we still have to prove the last part.
We assume that $f$ is a contraction of a $(K_{X_i} + D_i)$-negative extremal ray $\bR_{\ge 0}[\ell]$.
Choose a suitable ample divisor $H$ such that
$$(K_{X_i} + D_i + \eps H) \cdot \ell = 0 \, \textrm{ and } \, (X_i, D_i + \eps H) \textrm{ is still klt}$$
for some $0 < \eps \ll 1$.
By the cone theorem in MMP (cf.~\cite{KM98} or \cite[Theorem 1.1]{Fujino11}),
there is an $\bR$-Cartier divisor $\Theta_i$ on $V_i$ such that $K_{X_i} + D_i + \eps H = f^*\Theta_i.$
By the projection formula, $\Theta_i = K_{V_i} + f_*D_i + \eps f_*H$.
Then $(V_i, f_*D_i + \eps f_*H)$ is a klt pair.
So Lemma \ref{periodic-lemma-HM-RCC} implies the last part.
We have proved Lemma \ref{periodic-lemma-Null}.
\end{proof}

The following lemma exposes the relationship among the irreducible components of these $\Per_+(X_i,G)$.
We will also use this lemma to prove Theorem \ref{periodic-thm-ThmB'} (\ref{periodic-thm-ThmB_2}) later.

\begin{lemma} \label{periodic-lemma-comp}
Under the assumption of Lemma \ref{periodic-lemma-Null},
for any $0\le i\le s$, every non-uniruled irreducible component of $\Per_+(X_i,G)$ is
$G$-equivariant birational to some irreducible component of $\Per_+(X_{i+1},G)$ by $\tau_i$,
which is then isomorphic at the generic point of that irreducible component.
\end{lemma}

\begin{proof}
We use the same notation as in Lemma \ref{periodic-lemma-Null}. Let $Z^i$ be any non-uniruled irreducible component of $\Per_+(X_i,G)$.

If $\tau_i$ is a divisorial contraction for some $0\le i< s$ or $\tau_s$, by Lemma \ref{periodic-lemma-Null} (1) above,
$Z^i$ is not contained in the exceptional locus of $\tau_i$.
Hence $Z^i$ is $G$-equivariant birational to its birational transform in $X_{i+1}$,
and the latter is also an irreducible component of $\Per_+(X_{i+1},G)$.

If $\tau_i$ is a flip for some $0\le i< s$, by Lemma \ref{periodic-lemma-Null} (2), $Z^i$ is not contained
in the exceptional locus of $f \colon X_i \to V_i$.
Hence $Z^i$ is $G$-equivariant birational to its birational transform in $V_i$,
and the latter one is $G$-equivariant birational to its proper transform in $X_{i+1}$ via the map $f^+ \colon X_{i+1} \to V_i$.
This last one in $X_{i+1}$ is also the birational transform of $Z^i$ via the birational map $\tau_i \colon X_{i} \ratmap X_{i+1}$,
and hence an irreducible component of $\Per_+(X_{i+1},G)$.
In the above argument, we use the fact that both exceptional loci of $f \colon X_i \to V_i$ and $f^+ \colon X_{i+1} \to V_i$ lie over the same Zariski closed subset $\Delta_i \subset V_i$.
\end{proof}

For a projective variety $V$, we take a resolution $\widetilde{V} \to V$ and define the {\it Albanese map}
\index{Albanese!Albanese map} \index{Albanese!Albanese variety}
$$\xymatrix{\alb_V \colon V \ar@{-->}[r] & \Alb(V) \coloneqq \Alb(\widetilde V)}$$
as the natural composition $\xymatrix{V \ar@{-->}[r] & \widetilde V \ar[r]^(0.4){\alb_{\widetilde V}} & \Alb(\widetilde V).}$
It is known that $\alb_V$ is a well-defined morphism when $V$ has at worst rational singularities.

\begin{proof}[Proof of Theorem \ref{periodic-thm-ThmA'}]
By Proposition \ref{periodic-prop-finite-null}, replacing $G$ by a finite-index subgroup, we may assume that $(X,G)$ satisfies \ref{periodic-Hyp-A}.
So we can apply Proposition \ref{periodic-prop-Zhang-klt} by choosing $D=0$.
Then our assertion (\ref{periodic-thm-ThmA_1}) is just Proposition \ref{periodic-prop-Zhang-klt} (5).

\begin{proof}[Proof of Assertion (\ref{periodic-thm-ThmA_2})] \renewcommand{\qedsymbol}{}
We are going to prove this assertion by the backward induction on the index $i$ of $X_i$.
We will use the sequence \eqref{periodic-eqn-seq1} of $G$-equivariant birational maps as in Proposition~\ref{periodic-prop-Zhang-klt} with $D = 0$.
Recall that for $0\le i\le s+1$, $A_i$ (an $\bR$-Cartier divisor) denotes the push-forward of $A$ on $X_i$,
where $A = \sum L_i$ is a nef and big $\bR$-Cartier divisor as constructed in Lemma \ref{periodic-lemma-big-A}.
By Proposition \ref{periodic-prop-Zhang-klt} (5), we know that $A_i|_Z \num 0$ for every positive-dimensional $G$-periodic proper subvariety $Z$ of $X_i$.
Replacing $G$ by a finite-index subgroup, we may assume that $G$ stabilizes every irreducible component of $\Per_+(X_i,G)$.

Let $Z$ be an irreducible component of $\Per_+(Y,G)$.
By Proposition \ref{periodic-prop-Zhang-klt} (4) (with $D = 0$ always in the current theorem),
we know that $K_Y + A_Y$ is an ample $\bR$-Cartier divisor on $Y$.
Then $K_Y|_Z$ is also an ample $\bR$-Cartier divisor on $Z$ since $A_Y|_Z\num 0$.
Assume further that $Z$ is non-uniruled.
Then by Lemma \ref{periodic-lemma-finite-auto} applied to $H \coloneqq K_Y|_Z$, we know that $G|_{Z}$ is finite.
Hence a finite-index subgroup of $G$ fixes $Z$ pointwise.
So the assertion (\ref{periodic-thm-ThmA_2}) holds true on $Y = X_{s+1}$.

By induction we assume that for any irreducible component $Z^{i+1}$ of $\Per_+(X_{i+1},G)$,
either $Z^{i+1}$ is uniruled, or a finite-index subgroup of $G$ fixes $Z^{i+1}$ pointwise.
Now we choose any irreducible component $Z^i$ of $\Per_+(X_i,G)$.
Assume further that this $Z^i$ is non-uniruled.
Then by Lemma \ref{periodic-lemma-comp}, $Z^i$ is $G$-equivariant birational to its birational transform in $X_{i+1}$ by $\tau_i$,
and the latter is also an irreducible component of $\Per_+(X_{i+1},G)$.
By the inductive hypothesis, a finite-index subgroup of $G$ fixes that latter birational transform of $Z^i$, and then it also fixes $Z^i$ pointwise.
This proves the assertion (\ref{periodic-thm-ThmA_2}).
\end{proof}

\begin{proof}[Proof of Assertion (\ref{periodic-thm-ThmA_3})] \renewcommand{\qedsymbol}{}
The first part of the assertion (\ref{periodic-thm-ThmA_3}) has been proved by Lemma \ref{periodic-lemma-rho}.
If $\rho(X) = n$, the same lemma also tells us that $K_X$ is numerically trivial. Then $K_X$ is pseudo-effective.
Thus the second part follows from \cite[Theorem 2.4]{Zhang16-TAMS} (under the condition (ii) there).
\end{proof}

\begin{proof}[Proof of Assertion (\ref{periodic-thm-ThmA_4})] \renewcommand{\qedsymbol}{}
By Lemma \ref{periodic-lemma-A}, we only need to consider the case that $X$ is not an abelian variety.
We first assume that the irregularity $q(X) > 0$.
By Hironaka \cite{Hironaka77}, we can take an $\Aut(X)$-equivariant resolution $\pi \colon \widetilde X\to X$.
Then $q(\widetilde X) = q(X) > 0$ because
$X$ has only klt and hence rational singularities (cf.~\cite[Theorem 5.22]{KM98}).
By \cite[Lemma 2.13]{Zhang09-Invent}, $\alb_{\widetilde X}$ is a (necessarily $\Aut(\widetilde{X})$-equivariant) surjective birational morphism.
Hence the same holds for $\alb_X$ because $X$ has only rational singularities.
Note that for any $G$-periodic prime divisor $D$ on $X$, the image $\alb_X(D)$ of $D$ is a $G$-periodic subvariety of the abelian variety $\Alb(X)$.
It follows from Lemma \ref{periodic-lemma-A} again that such $D$ is $\alb_X$-exceptional.
Then we get the upper bound of distinct $G$-periodic prime divisors by Proposition \ref{periodic-prop-PropB}.
Next we assume that $q(X) = 0$.
Suppose that $X$ has $s$ distinct $G$-periodic prime divisors $B_1,\dots,B_s$.
Then the upper bound of $s$ has also been given by Proposition \ref{periodic-prop-PropB}.
This proves the assertion (\ref{periodic-thm-ThmA_4}).
\end{proof}

We have completed the proof of Theorem \ref{periodic-thm-ThmA'}.
\end{proof}

\begin{proof}[Proof of Theorem \ref{periodic-thm-ThmB'}]
Take a $G$-equivariant log resolution $\pi \colon \widetilde X\to X$ for the pair $(X,D)$,
and denote by $\widetilde D$ the proper transform of $D$.
Note that $\widetilde D$ is still a $G$-periodic non-uniruled prime divisor.
Replacing $G$ by a finite-index subgroup, we may assume that $\widetilde D$ is stabilized by $G$ and $(\widetilde X,G)$ satisfies \ref{periodic-Hyp-A}
(see Proposition \ref{periodic-prop-finite-null} and \cite[Lemma 3.1]{Zhang16-TAMS}).
Replacing $(X,D)$ by $(\widetilde{X},\widetilde D)$, it suffices to show Theorem \ref{periodic-thm-ThmB'} for the dlt case.
(We show the assertion (\ref{periodic-thm-ThmB_2}) for instance.
Suppose that $X$ has a $G$-periodic irreducible subvariety $Z_2$ different from $D$.
Then the $\pi$-proper transform ${\widetilde Z}_2$ of $Z_2$ is an irreducible component of
$\Per_+(\widetilde X,G)$ different from $\widetilde{D}$,
so it is uniruled. Hence $Z_2$ is uniruled.)

\begin{proof}[Proof of Assertion (\ref{periodic-thm-ThmB_1})] \renewcommand{\qedsymbol}{}
The surface case has been dealt with by Lemma \ref{periodic-lemma-surface}, so we only consider the case $n\ge 3$.
If $X$ were not rationally connected, then replacing $G$ by a finite-index subgroup, $X$ is $G$-equivariant birational to a quasi-\'etale torus quotient (cf.~\cite[Theorem 2.4]{Zhang16-TAMS}).
On the other hand, by the affirmative answer to Question \hyperref[periodic-qestion-QnA]{\ref{periodic-qestion-QnA} (2)} (i.e., Proposition \ref{periodic-prop-PropA}), $X$ is not $G$-equivariant birational to a quasi-\'etale torus quotient.
This is a contradiction.
\end{proof}

\begin{proof}[Proof of Assertion (\ref{periodic-thm-ThmB_3})] \renewcommand{\qedsymbol}{}
The dlt assumption implies that $X$ is also klt, so we may apply Theorem \ref{periodic-thm-ThmA'} to the pair $(X,G)$.
Then the assertion (\ref{periodic-thm-ThmB_3}) is a direct consequence of Theorem \ref{periodic-thm-ThmA'} (\ref{periodic-thm-ThmA_2}),
since $D$ is a $G$-periodic non-uniruled prime divisor.
\end{proof}

\begin{proof}[Proof of Assertion (\ref{periodic-thm-ThmB_4})] \renewcommand{\qedsymbol}{}
By the assertion (\ref{periodic-thm-ThmB_1}) above, we can apply Lemma \ref{periodic-lemma-ps} and say that $K_X+D$ is pseudo-effective.
This in turn allows us to apply Proposition \ref{periodic-prop-Zhang-dlt} to the dlt pair $(X,D)$.
Note that the $G$-equivariant birational map $X\ratmap Y$ is originally constructed in Proposition \ref{periodic-prop-Zhang-klt}
for the klt pair $(X, (1-\eps)D)$ with $\eps > 0$ sufficiently small.
Then the assertion (\ref{periodic-thm-ThmB_4}) comes readily from Proposition \ref{periodic-prop-Zhang-dlt} (3).
\end{proof}

\begin{proof}[Proof of Assertion (\ref{periodic-thm-ThmB_5})] \renewcommand{\qedsymbol}{}
We first prepare the following for the proof of this assertion.
Note that $(X,D)$ is dlt, then $(X,D_\eps)$ is klt,
where $D_\eps\coloneqq(1-\eps)D$ for some $0<\eps\ll 1$ (cf.~\cite[Proposition 2.41]{KM98}).
So we can apply Proposition \ref{periodic-prop-Zhang-klt} to the klt pair $(X, D_\eps)$.
Then there is a sequence $\tau_s \circ \cdots \circ \tau_0$ of $G$-equivariant birational maps:
\begin{equation}\label{periodic-eqn-seq2} \tag{$\star\star$}
X=X_0 \overset{\tau_0}{\ratmap} X_1 \overset{\tau_1}{\ratmap} \cdots \overset{\tau_{s-1}}{\ratmap} X_s \overset{\tau_s}{\longrightarrow} X_{s+1}=Y
\end{equation}
such that each $\tau_j \colon X_j \ratmap X_{j+1}$ for $0\le j < s$ is either a divisorial contraction of a $(K_{X_j} + D_{\eps, j})$-negative extremal ray or a $(K_{X_j} + D_{\eps, j})$-flip;
the $\tau_s \colon X_s\to X_{s+1}=Y$ is a birational morphism such that
$K_{X_s}+D_{\eps, s} = \tau_s^*(K_Y+D_{\eps, Y});$
here $D_{\eps, i}\subset X_i$ for $0\le i\le s+1$ denotes the push-forward of $D_\eps$.
It follows from \cite[Corollaries 3.42 and 3.43]{KM98} that each $(X_i,D_{\eps, i})$ for $0\le i\le s$ is klt.
So $(Y,D_{\eps,Y})$ is also klt.
In particular, by Lemma \ref{periodic-lemma-HM-RCC}, each $D_{\eps, i}$ for $0\le i\le s+1$ is indeed a divisor since $D$ is non-uniruled.

Now the first part of the assertion (\ref{periodic-thm-ThmB_5}), i.e., $K_Y + D_Y \sim_\bQ 0$, follows from Proposition \ref{periodic-prop-Zhang-dlt} (2).

By the first part we have proved and Proposition \ref{periodic-prop-Zhang-klt} (4), we know that
$$-\eps D_Y + A_Y \sim_\bQ K_Y + (1-\eps) D_Y + A_Y$$
is an ample $\bR$-Cartier divisor.
Note also that $A_Y$ is $\bR$-Cartier by Proposition \ref{periodic-prop-Zhang-klt} (2), and then so is $D_Y$.
Hence by Lemma \ref{periodic-lemma-Q-Cartier}, $D_Y$ is $\bQ$-Cartier, and then so is $K_Y$.

Note that $Y$ is rationally connected (since so is $X$) and $D_Y$ is a non-uniruled divisor.
Hence by Lemma \ref{lemma-can-sing} (i), $K_Y+D_Y\sim_\bQ 0$ implies that $(Y, D_Y)$ has only canonical singularities (and $D_Y$ is a normal variety),
so does $Y$ (cf.~\cite[Corollary 2.35]{KM98}).
This proves the assertion (\ref{periodic-thm-ThmB_5}).
\end{proof}

\begin{proof}[Proof of Assertion (\ref{periodic-thm-ThmB_6})] \renewcommand{\qedsymbol}{}
By the adjunction theorem for dlt pairs (cf.~\cite[Proposition 3.9.2]{Fujino07} or \cite[\S16 and \S17]{Kollar92}),
\index{adjunction formula!for dlt pair}
there exists an effective divisor $\Diff_{D_Y}(0)$ on $D_Y$ such that
$$K_{D_Y} + \Diff_{D_Y}(0) = (K_Y + D_Y)|_{D_Y} \sim_\bQ 0.$$
Note that $D_Y$ itself (as a variety) is non-uniruled and normal.
Then by applying Lemma \ref{lemma-can-sing} (ii) to the pair $\big(D_Y, \Diff_{D_Y}(0)\big)$,
we have $\Diff_{D_Y}(0) = 0$ and $D_Y$ has at worst canonical singularities.
Thus $K_{D_Y}\sim_\bQ 0$. This proves the assertion (\ref{periodic-thm-ThmB_6}).
\end{proof}

\begin{proof}[Proof of Assertion (\ref{periodic-thm-ThmB_7})] \renewcommand{\qedsymbol}{}
By the assertion (\ref{periodic-thm-ThmB_4}) we have proved, every positive-dimensional $G$-periodic subvariety of $Y$ is contained in $D_Y$, so $\Per_+(Y,G)=D_Y$.
In particular, by Proposition \ref{periodic-prop-Zhang-klt} (5), we have $A_Y|_{D_Y} \num 0$.
We already see in the proof of the assertion (\ref{periodic-thm-ThmB_5}) that $-\eps D_Y+A_Y$ is an ample $\bR$-Cartier divisor,
and then so is $(-\eps D_Y+A_Y)|_{D_Y} \num -\eps D_Y|_{D_Y}$.
Note that by the assertion (\ref{periodic-thm-ThmB_5}), $D_Y$ is $\bQ$-Cartier.
The assertion (\ref{periodic-thm-ThmB_7}) follows.
\end{proof}

\begin{proof}[Proof of Assertion (\ref{periodic-thm-ThmB_2})] \renewcommand{\qedsymbol}{}
Suppose to the contrary that some $Z_k$ with $k\ge 2$ is non-uniruled.
Note that in our proof of the assertion (\ref{periodic-thm-ThmB_5}), we applied Proposition \ref{periodic-prop-Zhang-klt} to the klt pair $(X,D_\eps)$
and produced the sequence \eqref{periodic-eqn-seq2} of $G$-equivariant birational maps.
So by Lemma \ref{periodic-lemma-comp}, such $Z_k$ is $G$-equivariant birational to some irreducible component of $\Per_+(Y,G)$
by $\tau_{s}\circ \cdots \circ\tau_0$,
which is isomorphic at the generic point of $Z_k$.
On the other hand, the assertion (\ref{periodic-thm-ThmB_4}) says that $\Per_+(Y,G)=D_Y$ has only one irreducible component.
So such $Z_k$ is birational to $D_Y$. By the irreducibility of $Z_k$ we know that $Z_k$ coincides with $D = Z_1$, which is a contradiction.
This ends the proof of the assertion (\ref{periodic-thm-ThmB_2}).
\end{proof}

We have completed the proof of Theorem \ref{periodic-thm-ThmB'}.
\end{proof}

\begin{remark}
With the assumption and notation in Theorem \ref{periodic-thm-ThmB'}, we have:
\begin{enumerate}[(1)]
  \item It follows from Proposition \ref{periodic-prop-PropB} with $B_1\coloneqq D$ that the Picard number $\rho(X) \ge n+1$.
  \item Note that the positive-dimensional part of the singular locus $\Sing Y$ of $Y$ is contained in $D_Y$ by Theorem \ref{periodic-thm-ThmB'} (\ref{periodic-thm-ThmB_4}).
  So we have $\dim(\Sing Y) \le \max\{0,\dim Y-3\}.$
  Indeed, by Theorem \ref{periodic-thm-ThmB'} (\ref{periodic-thm-ThmB_5}), $(Y, D_Y)$ has at worst canonical singularities.
  After $(\dim Y-2)$-times hyperplane cutting as in \cite[Corollary 5.18]{KM98}, we reach a canonical surface
  pair $(S, D_S)$ (cf.~\cite[Lemma 5.17 (1)]{KM98}).
  So by \cite[Theorem 4.5]{KM98}, $D_S \cap \Sing S = \varnothing$,
  and hence $Y$ is smooth at its codimension-$2$ points lying inside $D_Y$.
  This shows that $\dim (D_Y \cap \Sing Y) \le \max\{0,\dim Y-3\}$.
  \item Suppose $\dim Y=2$. Then $Y$ is smooth in a neighborhood of $D_Y$, and $D_Y$ is a (smooth) elliptic curve,
  since $D_Y$ is normal and $K_{D_Y}\sim_\bQ 0$.
  \item Suppose $\dim Y=3$. Then $Y$ has at worst isolated singularities.
  Further, $K_{D_Y} \sim_{\bQ} 0$ implies that $D_Y$ is either a smooth abelian surface or hyperelliptic surface,
  or a normal $K3$ surface or Enriques surface with at worst Du Val singularities.
\end{enumerate}
\end{remark}

\section{Proofs of Theorem \ref{periodic-thm-ThmC} and Proposition \ref{periodic-prop-finite-periodic}}\label{periodic-section-proofs2}

\begin{proof}[Proof of Theorem \ref{periodic-thm-ThmC}]
(1) $\Longrightarrow$ (2) is proved by Lemma \ref{periodic-lemma-A}.

(2) $\Longrightarrow$ (1) comes from \cite[Theorem 2.4 or Lemma 3.10]{Zhang16-TAMS}.

(1) $\Longrightarrow$ (3) is true by choosing $D'=0$,
and note that quotient singularities are $\bQ$-factorial klt and $K_{X'} \sim_{\bQ} 0$.

(3) $\Longrightarrow$ (1) follows from \cite[Theorem 2.4]{Zhang16-TAMS} (under the condition (ii) there).

(1) $\Longrightarrow$ (4) is just our Proposition \ref{periodic-prop-PropA} (1).
\end{proof}

Finally, we shall prove Proposition \ref{periodic-prop-finite-periodic}.
But prior to that, we give two lemmas to deal with the abelian variety case.
It should be noted that even for the abelian variety case, unfortunately,
we have not been able to generalize Proposition \ref{periodic-prop-finite-periodic} (\ref{periodic-prop-finite-periodic_1}) to higher dimensions.

\begin{lemma} \label{periodic-lemma-big-general-type}
Let $X$ be an abelian variety of dimension $n\ge 2$ and $D$ a prime divisor.
Then $D$ as an algebraic variety is of general type if and only if $D$ is a big divisor.
\end{lemma}

\begin{proof}
Suppose that $D$ as an algebraic variety is of general type.
Note that $D$ is a Cartier divisor on smooth $X$ so that the canonical divisor $K_D$ is a well-defined Cartier divisor.
Let $\nu \colon D^\nu \to D$ be the normalization of $D$ and $C$ the conductor of $\nu$ which is an effective divisor on $D^\nu$.
Then $$K_{D^\nu} + C = \nu^*K_D.$$
Also, by the (generalized) adjunction formula in \cite[Proposition 5.73]{KM98} and $K_X = 0$, we have
\index{adjunction formula!general version}
$$K_D = (K_X + D)|_D = D|_D.$$
Take a log resolution $\mu \colon D'\to D^\nu$ for the pair $(D^\nu, C)$.
We then have
$$K_{D'} + \mu^{-1}_*C + E_1 = \mu^*(K_{D^\nu} + C) + E_2,$$
where $E_1$ and $E_2$ are effective $\mu$-exceptional divisors, and $K_{D'}$ is big since $D$ is of general type.
Hence by the above three equalities we can show that
\begin{align*}
  & c \cdot m^{n-1} < h^0(D', mK_{D'}) \le h^0\big(D', m(K_{D'} + \mu^{-1}_*C + E_1)\big) \\
  & = h^0\big(D', m\mu^*(K_{D^\nu} + C) + mE_2\big) = h^0\big(D^\nu, m(K_{D^\nu} + C)\big) \\
  & = h^0\big(D^\nu, m\nu^*K_D\big) = h^0\big(D^\nu, m\nu^*(D|_D)\big),
\end{align*}
for some $c>0$ and $m\gg 1$.
The first inequality comes from the definition of big Cartier divisor,
and the second equality holds by the projection formula for the morphism $\mu \colon D'\to D^\nu$.
It follows that the Cartier divisor $\nu^*(D|_D)$ is big, then so is $D|_D$ since $\nu$ is birational
(or just by the definition of big Cartier divisors on non-normal varieties).
Hence $D|_D$ is nef and big, so $0 < (D|_D)^{n-1} = D^n$ (cf.~\cite[Proposition 2.61]{KM98}).
Note that $D$ on $A$ is nef. So $D$ is big.

Conversely, suppose that $D$ is a big divisor (and contains the origin point).
We have seen that $D$ is ample by Lemma \ref{periodic-lemma-eff-on-torus}.
Then it is well-known that $K(\cO_X(D))\coloneqq \Ker\phi_{\cO_X(D)}$ is finite (see e.g. \cite[Proposition 4.5.2]{BL04}),
where $\phi_{\cO_X(D)}$, the canonical map from $X$ to its dual abelian variety $\widehat{X} \coloneqq \Pic^0(X)$,
is defined as following
$$\phi_{\cO_X(D)} \colon X\to \widehat{X},\ \ x \mapsto T_x^*\cO_X(D) \otimes \cO_X(D)^{-1} = \cO_X(T_x^*D - D).$$
On the other hand, by \cite[Theorem 10.9]{Ueno75} or \cite[\S15.7]{BL04}, we also know that if denote by
$$Z\coloneqq \{x\in X : x + D \subset D\}^0$$
the identity connected component of the stabilizer of $D$ in $X$,
then $Z$ is an abelian subvariety of $X$ contained in $D$
such that the quotient variety $D/Z$ is of general type.
Note that in our situation, $Z$ is contained in the finite group scheme $K(\cO_X(D))$ and hence equals $0$.
It follows that $D$ is of general type which finishes the other direction of Lemma \ref{periodic-lemma-big-general-type}.
\end{proof}

\begin{lemma} \label{periodic-lemma-no-g-periodic}
Let $X$ be an abelian variety of dimension $n=2$ or $3$,
and $G\le \Aut_{\rm variety}(X)$ such that $G\isom \bZ^{\oplus n-1}$ is of positive entropy.
Then for any non-trivial $f\in G$, there is no $f$-periodic prime divisor.
\end{lemma}

\begin{proof}
If $X$ is an abelian surface admitting an automorphism $f$ of positive entropy
and $C$ is an irreducible $f$-periodic curve, then by \cite[Lemma 10.1]{Ueno75}, $\kappa(C)\ge 0$ and hence $C$ is irrational,
which contradicts Lemma \ref{periodic-lemma-surface}.
Thus we only need to consider the case $X$ is an abelian $3$-fold.

Suppose to the contrary that there exists at least one $f$-periodic prime divisor $D$.
Replacing $f$ by some power, we may assume that $f(D) = D$.
Write $f = T_a \circ g$ with $T_a$ a translation and $g$ a group automorphism (fixing the origin point).
Also, after replacing $f$ by $T_{-d}\circ f \circ T_d$, $G$ by $T_{-d}\circ G \circ T_d$, and $D$ by $T_{-d}(D) = D - d$ for some $d\in D$,
we may assume that $D$ contains the origin point.
According to the Kodaira dimension of $D$, we have the following three cases.

Case 1): $\kappa(D) = \dim D = 2$, i.e., $D$ as an algebraic variety is of general type.
Then $D$ is big by Lemma \ref{periodic-lemma-big-general-type}.
It follows that $f$ is of null entropy by Lemma \ref{periodic-lemma-max-ev} (see also \cite[Lemma 2.23]{Zhang09-JDG}),
contradicting our assumption on $f$.

Case 2): $\kappa(D) = 1$. By \cite[Theorem 10.9]{Ueno75} or \cite[\S15.7]{BL04},
the identity connected component $Z$ of the stabilizer of $D$ in $X$, i.e.,
$Z\coloneqq \{x\in X : x + D \subseteq D\}^0,$
is an abelian subvariety of $X$ containing in $D$,
such that $\dim Z = \dim D - \kappa(D) = 1$ and $D/Z$ is of general type.
As in \cite[Lemma 2.11]{Zhang10-AIM223}, we can prove that $g(Z) = Z$ and hence $\pi \colon X\to X/Z$ is an $f$-equivariant
(and also $g$-equivariant) fibration with a fibre $Z$.
Indeed, for any $z\in Z$, we have
$$g(z) + D = g(z) + f(D) = g(z) + (a + g(D)) = a + g(z + D) \subseteq a + g(D) = f(D) = D.$$
So $g(Z) = Z$ because $g$ is a group automorphism.
Using the main result of Dinh--Nguy{\^e}n \cite{DN11}, we have
$d_1(f|_X) = d_1(g|_X) = \max\{d_1(g|_{X/Z}), d_1(g|_\pi)\},$
where $d_1(g|_\pi)$ denotes the relative dynamical degree in their sense.
Note that the $g$-action on $X/Z$ has a fixed point, the origin point.
So by \cite[Remark 3.4]{DN11}, $d_1(g|_\pi) = d_1(g|_Z)$ and hence equals $1$ since $\dim Z = 1$.
On the other hand, $f|_{X/Z}$ stabilizes the nef and big divisor $D/Z$
(bigness comes from Lemma \ref{periodic-lemma-big-general-type} again).
Hence $d_1(g|_{X/Z}) = d_1(f|_{X/Z}) = 1$.
Overall, we have shown that $d_1(f|_X) = 1$, which is a contradiction.

Case 3): $\kappa(D)\le 0$.
Then $\kappa(D) = 0$ and $D = \delta + Z$ is a translation of some abelian subvariety $Z$ of $X$ (cf.~\cite[Theorem 10.3]{Ueno75}).
Now we have $\delta + Z = D = f(D) = a + g(D) = a + g(\delta) + g(Z)$.
Since $g$ is a group automorphism and $Z$ contains the origin point, $a + g(\delta) - \delta \in Z$ and hence $g(Z) = Z$.
Then consider the quotient map $\pi \colon X \to X/Z$ which is a $g$-equivariant (and also $f$-equivariant) fibration.
As in Case 2), we also have the following equalities concerning the first dynamical degrees
$d_1(f|_X) = d_1(g|_X) = \max\{d_1(g|_{X/Z}), d_1(g|_\pi)\} = d_1(g|_Z),$
here $d_1(g|_{X/Z}) = 1$ because $Z$ is an abelian surface and hence $\dim X/Z = 1$.
Thus we have deduced that $d_1(g|_Z)>1$ since $f$ is of positive entropy.

Write $G=\langle f_1, f_2\rangle$ and $f_i = T_{a_i} \circ g_i$ for group automorphisms $g_i$.
Consider the induced composite morphisms $\pi_i \colon g_i(Z) \hookrightarrow X \to X/Z$.
Suppose that for each $i$, $\dim \im \pi_i = 0$, i.e., $\im \pi_i$ is the origin point of the elliptic curve $X/Z$.
So $g_i(Z) = Z$ and then it follows that $\pi \colon X\to X/Z$ is $G$-equivariant, contradicting with \cite[Lemma 2.10]{Zhang09-Invent}.
Therefore, we may assume that $\pi_1$ dominates $X/Z$ and hence it is flat by \cite[Proposition 9.7]{Hartshorne}.
Moreover, every irreducible component of the geometric fibre of $\pi_1$ over $0$ (i.e., $g_1(Z)\cap Z$) has dimension $1$.
Let $F$ be any such irreducible component.
It is easy to see that this $F$ is a $g$-periodic curve in $Z$ (since $f_i$ commutes with $f$ implies $g_i$ commutes with $g$).
However, as we have seen, an abelian surface can not contain any $g$-periodic curve for any automorphism $g$ of positive entropy.
So we also derive a contradiction in this case and hence finish the proof of Lemma \ref{periodic-lemma-no-g-periodic}.
\end{proof}

\begin{remark}
Under the same conditions in Lemma \ref{periodic-lemma-no-g-periodic},
we can show that there is no (irreducible) $f$-periodic curve either.
Actually, let $C$ be an irreducible $f$-periodic curve on an abelian $3$-fold $X$.
If $\kappa(C) = 1$, i.e., $C$ is of general type and hence $\Aut(C)$ is finite.
So $f^m|_C = \id_C$ for some $m > 0$. Write $f = T_a \circ g$ as usual.
It follows from \cite[Lemma 13.1.1]{BL04} that the identity component $Z$ of the pointwise fixed point set $X^{g^m}$ is a positive-dimensional abelian subvariety.
Then $X\to X/Z$ is a $G$-equivariant fibration (see also \cite[Lemma 2.11]{Zhang10-AIM223}),
which contradicts with \cite[Lemma 2.10]{Zhang09-Invent}.
If $\kappa(C) = 0$, then $C = \delta + E$ for some abelian subvariety $E$ of dimension $1$.
Consider the quotient map $\pi \colon X \to X/E$ which is a $g$-equivariant fibration.
As in the proof of Lemma \ref{periodic-lemma-no-g-periodic}, we may assume that the dimension of $\pi(g_1(E))$ is $1$,
i.e., $\pi(g_1(E))$ is a $g$-periodic curve in the abelian surface $X/E$.
Note also that $d_1(g|_{X/E}) = d_1(g|_X) > 1$ by Dinh--Nguy{\^e}n \cite{DN11}.
This contradicts the surface case of Lemma \ref{periodic-lemma-no-g-periodic}.
\end{remark}

\begin{proof}[Proof of Proposition \ref{periodic-prop-finite-periodic}]

The assertion (\ref{periodic-prop-finite-periodic_1}) has been proved by Lemma \ref{periodic-lemma-no-g-periodic}.
To prove the assertion (\ref{periodic-prop-finite-periodic_2}), we first consider the case that the irregularity $q(X) > 0$.
Then the Albanese map $\alb_X\colon X\to \Alb(X)$ is a $G$-equivariant birational surjective morphism
by the maximality of the dynamical rank of $G$ (cf.~\cite[Lemma 2.13]{Zhang09-Invent}).
So for any $g$-periodic prime divisor $D$ on $X$, one has $\alb_X(D)$ is a $g$-periodic subvariety of $\Alb(X)$.
However, according to Lemma \ref{periodic-lemma-no-g-periodic},
$\alb_X(D)$ can not be a $g$-periodic divisor, i.e., $D$ is an $\alb_X$-exceptional divisor.
Hence for any $1\le i\le n-1$, it follows from the commutativity of $G$ that each $g_i(D)$ is also a $g$-periodic $\alb_X$-exceptional divisor.
Note that for a birational morphism, there are only finitely many (irreducible) exceptional divisors.
Thus $D$ is $g_i$-periodic for any $i$ and hence $G$-periodic.
Then by Proposition \ref{periodic-prop-PropB},
there are at most $\rho(X) - n$ distinct $g$-periodic prime divisors.

Next, we may assume that the irregularity $q(X) = 0$.
This also holds for any resolution of $X$ because $X$ has only klt and hence rational singularities (cf.~\cite[Theorem 5.22]{KM98}).
We only need to prove the claim that
there are only finitely many $g$-periodic prime divisors $D_j$ with $1\le j\le k$ for some $k > 0$.
Indeed, assuming this claim for the time being,
as in the case $q(X) > 0$, we can show that $D_j$ is $G$-periodic for any $j$.
Then by Proposition \ref{periodic-prop-PropB}, we would have the upper bound $\rho(X) - n$.
For this claim, the surface case is well known.
Actually, it follows from the Hodge index theorem and the fact that
every $g$-periodic curve is perpendicular to the nef and big divisor $A\coloneqq L_g + L_{g^{-1}}$
as in the proof of Lemma \ref{periodic-lemma-surface},
where $L_{g^{\pm 1}}$ are the nef divisors corresponding to the first dynamical degree $d_1(g^{\pm 1})$ of $g^{\pm 1}$.
Therefore, we still have to prove this claim for the case $n = 3$.

Suppose to the contrary that the above claim does not hold.
Namely, there are infinitely many distinct $g$-periodic prime divisors $D_j$ with $j\ge 1$. Let
$$\kappa \coloneqq \kappa \big(X,\sum_{j=1}^{r}D_j\big)
= \max\bigg\{ \kappa \big(X,\sum_{j=1}^{t}D_j\big) : D_j \text{ is $g$-periodic, } t\ge 1 \bigg\}$$
for some $r\ge 1$ and denote $E_0 \coloneqq \sum_{j=1}^{r}D_j$.
Replacing $g$ by its power, we may assume that $g(D_j) = D_j$ for all $j\le r$.
As reasoned in Proposition \ref{periodic-prop-PropB}, we have $\kappa\ge 1$.

For any $1\le i\le n-1$, let $E_i\coloneqq g_i^*E_0$.
It is easy to see that $E_i$ is also $g$-periodic since $g$ commutes with each $g_i$,
and hence $\kappa(X,E_i) = \kappa(X,E_0 + E_i) = \kappa$ by the maximality of $\kappa$.
Replacing $E_0$ by some $mE_0$, we may assume that the dominant rational map
$$\Phi_{|E_i|} \colon X \ratmap \Phi_{|E_i|}(X) \subseteq \bP H^0\big(X,\cO_X(E_i)\big)$$
is an Iitaka fibration associated to $E_i$ and its image has dimension equal to $\kappa$ for any $0\le i\le n-1$.
Take a $g$-equivariant resolution $\pi \colon X' \to X$ of $\Sing X$ and $\Bs(|E_i|)$ due to Hironaka \cite{Hironaka77},
such that the linear system $|\pi^*E_i| = |M_i| + F_i,$
where each $M_i$ is base point free, $F_i$ is the fixed component of $|\pi^*E_i|$, and their divisor classes are $g$-stable.
Now the morphism $\Phi_{|M_i|}$ is birational to $\Phi_{|E_i|}$.
Let $Y_i \to \Phi_{|M_i|}(X')$ be the normalization, and
$\phi_i \colon X' \to Y_i$
the induced morphism, which is an algebraic fibre space with connected fibres.
Denote by $A_i$ the ample divisor on $Y_i$ such that $M_i = \phi_i^*A_i$.
We have $\kappa(X',M_0 + M_i) = \kappa(X,E_0 + E_i) = \kappa$ by the maximality of $\kappa$.
Thus the free divisor $M_0 + M_i$ is the pullback of some ample divisor on a variety of dimension $\kappa$,
which implies that $(M_0 + M_i)^{\kappa + 1} = 0$.
In particular, $M_0^{\kappa} \cdot M_i = 0 = M_0 \cdot M_i^{\kappa}$.

We assert that $\kappa\le n-2=1$.
Indeed, by blowing up $Y_i$ and $X'$ further, we may assume that $Y_i$ is also smooth.
Replacing $\phi_i$ by the new morphism, the new $A_i$ on the new $Y_i$ is only nef and big.
Nevertheless, we obtain a $g$-equivariant fibration $\phi_i \colon X'\to Y_i$ of smooth varieties
such that $g$ preserves the nef and big divisor $A_i$ on $Y_i$.
It follows from \cite[Lemma 2.5]{Zhang10-AIM223} that $\kappa\le n-2=1$, thus $\kappa = 1$ in the present case.
(Remark: in what follows, the blowing up of $Y_i$ is unnecessary, since $Y_i$ is a normal and hence a smooth curve.
In particular, the divisor $A_i$ is still ample, and $\phi_i$ is flat and hence equidimensional; see \cite[Proposition~9.7]{Hartshorne}.
Indeed, the argument below works as long as $\phi_i$ is equidimensional.)

For $1\le i\le n-1$, let $C$ be any curve in a general fibre $F_i$ of $\phi_i$.
Take general ample divisors $H_j$ on $X'$ containing $C$ with $1 \le j < n-\kappa$.
Let $S \coloneqq H_1 \cap \cdots \cap H_{n-\kappa-1}$.
Then
$$0\le C \cdot M_0 = C \cdot M_0|_S \le M_i^{\kappa}|_S \cdot M_0|_S = M_i^{\kappa} \cdot M_0 \cdot H_1 \cdots H_{n-\kappa-1} = 0.$$
Thus $A_0 \cdot (\phi_0)_*C = 0$ by the projection formula.
So $\phi_0$ contracts $C$ (and hence the whole $F_i$) by the ampleness of $A_0$.
Then by the Rigidity Lemma \ref{periodic-lemma-rigidity}, $\phi_0 = t_i \circ \phi_i$ for some morphism $t_i \colon Y_i \to Y_0$.
Interchanging the role of $M_0$ with $M_i$, we get another morphism $s_i \colon Y_0 \to Y_i$ such that $\phi_i = s_i \circ \phi_0$.
Hence $\phi_i = s_i \circ t_i \circ \phi_i$.
The surjectivity of $\phi_i$ then implies that $s_i \circ t_i = \id$. Similarly, $t_i \circ s_i = \id$.
Thus $s_i$ and $t_i$ are isomorphisms and inverse to each other by the normality of $Y_i$.
Therefore, we can write $M_i = \phi_i^*A_i = \phi_0^*B_i$ with $B_i \coloneqq s_i^*A_i$ an ample divisor on $Y_0$.

Now the automorphism $g_i$ on $X$ descends to an isomorphism between the bases of the Iitaka fibrations $\Phi_{|E_0|}$ and $\Phi_{|E_i|}$,
while the latter two are birational to $\Phi_{|M_0|}$ and $\Phi_{|M_i|}$, respectively.
So $g_i$ induces an isomorphism from (the normalization of) $\Phi_{|A_0|}(Y_0)$ to (the normalization of) $\Phi_{|B_i|}(Y_0)$,
which is an automorphism of $Y_0$ now.
Thus $G$ acts on $Y_0$ bi-regularly.
Replacing $X \ratmap Y_0$ by a $G$-equivariant resolution $X''$ of the graph,
we have a non-trivial $G$-equivariant fibration between two smooth projective varieties.
Contradicts the maximal dynamical rank assumption on $G$ (cf.~\cite[Lemma 2.10]{Zhang09-Invent}).
This ends the proof of Proposition \ref{periodic-prop-finite-periodic}.
\end{proof}

\phantomsection
\addcontentsline{toc}{section}{Acknowledgments}
\noindent
\textbf{Acknowledgments. }
The second author is supported by NSFC and the Science Foundation of Shanghai (No.~13DZ2260400).
The third author is supported by an ARF of NUS; he would also like to thank the following institutes for the support and hospitality when preparing the paper:
Max Planck Institute for Mathematics in Bonn, IH\'ES in Paris, University of Bayreuth, and East China Normal University.
The authors would like to thank the referee for carefully reading our manuscript and his/her valuable suggestions which help improving the exposition of the paper.

\end{document}